%% file: cable_tree_thresholds.tex
\documentclass[11pt,a4paper]{scrartcl}
\usepackage{mathtools}
\usepackage{needspace}
\usepackage[mathfont=stix2]{mystyle-koma-enhanced}
\usepackage{tikz}
\usetikzlibrary{arrows.meta,decorations.pathreplacing}
\hypersetup{
 pdftitle={Loop soup percolation and boundary dimension on metric trees},
 pdfauthor={Christian M\"onch},
 pdfsubject={Critical intensities and Hausdorff dimension in the Green boundary metric}
}

\newcommand{\Pp}{\mathbb P}
\newcommand{\E}{\mathbb E}
\newcommand{\Pt}{\mathbf P}
\newcommand{\Et}{\mathbf E}
\newcommand{\wt}{\widetilde}
\newcommand{\ind}{\mathbf 1}
\newcommand{\capacity}{\operatorname{cap}}
\newcommand{\ac}{\widetilde\alpha_{\mathsf c}}
\newcommand{\Ga}{\operatorname{Gamma}}
\newcommand{\de}{\deg_T}
\title{Loop soup percolation\\
and boundary dimension on metric trees}
\author{Christian M\"onch}
\date{8 October 2026}

\begin{document}
\maketitle

\begin{abstract}
We determine the percolation threshold of the Brownian loop soup on
a transient locally finite weighted tree with proper cable metric
and no killing. It equals $1-D_G/2$, where $D_G\in[0,2]$ is the
Hausdorff dimension of the boundary in the metric induced by the
normalized Green function. This gives an analogue of Lyons' branching-number
formula for Bernoulli percolation, with Green distance replacing
graph distance. Spherically symmetric trees with unit conductances
and branching number two realize every threshold in $[1/2,1]$, while
geometric conductances give explicit positive thresholds below one-half
and a subdivided binary tree has threshold zero.
Under the capacity assumption of Drewitz, Pr\'evost and Rodriguez,
we characterize the threshold one-half by a finite-energy condition,
without a minimum-degree assumption. Non-percolation at one-half
then follows from their theorem and Lupu's coupling. For
supercritical Galton--Watson trees with unit conductances, conditioned
on survival, the boundary dimension is almost surely one and the
threshold is one-half, with non-percolation at criticality. No
offspring moment assumption is required, and leaves and vertices
with one child are allowed.
The classification also applies intrinsically to proper real trees
spanned by their infinite rays.
\end{abstract}

\section{Introduction and discussion of main results}
\label{sec:results}

A classical result of Lyons~\cite{Lyons90} identifies the critical probability
of independent bond percolation on a tree as the reciprocal of its
branching number. The branching number measures growth through
cutsets. Its logarithm is the Hausdorff dimension of the boundary of
infinite rays, with distance $e^{-n}$ between rays whose common
initial path has length $n$. Finite-energy flows connect these geometric and
probabilistic descriptions; see~\cite[Chapters~3, 5 and~15]{LP}.
We prove an analogous classification for Brownian loop soup
percolation, with the boundary geometry determined by the Green
function of the underlying diffusion.

Loop soup percolation concerns the connected components formed by
a Poisson ensemble of loops. Increasing the intensity adds loops,
and the central question is whether their intersections produce an
unbounded component. Brownian loop soups were introduced by Lawler
and Werner~\cite{LW}; their Markovian counterparts and occupation
fields are developed in Le Jan~\cite{LeJan}. The percolation problem
on discrete graphs was studied by Le Jan and Lemaire~\cite{LJL}
and Chang and Sapozhnikov~\cite{CS}.

Passing from a discrete graph to its cable system replaces each edge
by an interval and the random walk by a Brownian diffusion. Loops
can then intersect inside edges, and loops that visit at most one
vertex also contribute to connectivity. At intensity $1/2$, Lupu's
coupling~\cite{Lupu16} identifies the cable loop clusters with the
sign clusters of the Gaussian free field. Together with the results
of Drewitz, Pr\'evost and Rodriguez~\cite{DPR} and Chang, Du and
Li~\cite{CDL}, this explains the occurrence of the critical intensity
$1/2$ in many transient settings. However, these results do not determine
the critical intensity on arbitrary locally finite trees.

Our main result identifies the critical intensity as one minus half
the Hausdorff dimension of the boundary in the metric induced by
the normalized Green function; see Theorem~\ref{thm:classification}.
It applies to every transient locally finite weighted tree with
proper cable metric and no killing. The dimension is independent of
the root and of the weighted graph used to represent the metric tree.

Trees with unit conductances and branching number two realize every
threshold in $[1/2,1]$, geometric conductances give thresholds in
$(0,1/2]$, and a subdivided binary tree has threshold zero; see
Corollary~\ref{cor:fixed-branching} and
Propositions~\ref{prop:weighted-example} and~\ref{prop:zero-threshold}.
Theorem~\ref{thm:half-analytic}
separates the capacity condition giving non-percolation at one-half
from the finite-energy condition giving percolation above it.
For supercritical Galton--Watson trees with unit conductances,
conditioned on survival, the Green boundary dimension is almost surely
one and the threshold is one-half, without any offspring moment or lower
offspring assumption; see Theorem~\ref{thm:GW}.

\subsection{The model and the classification theorem}

Let $T=(V,E)$ be an infinite locally finite tree with conductances
$c_e\in(0,\infty)$ and no killing. Write $u\sim v$ when
$\{u,v\}\in E$ and $c_{uv}=c_{\{u,v\}}$. The associated random
walk has transition probabilities
\begin{equation}
 P(u,v)=\frac{c_{uv}}{\sum_{w\sim u}c_{uw}}\qquad(v\sim u).
 \label{eq:walk-transition}
\end{equation}
We call $(T,c)$ transient if this walk visits each vertex only
finitely many times almost surely. Its cable system $\wt T$ is
obtained by replacing each edge $e$ by an interval of length
$\ell_e=(2c_e)^{-1}$ and identifying the endpoints with the incident
vertices. We use the induced path metric. A metric space is
\emph{proper} if every closed bounded ball is compact. Our standing
assumptions for the deterministic results are
\begin{equation}
 (T,c)\text{ is transient, and }\wt T\text{ is proper.}
 \label{eq:weighted-setting}
\end{equation}
The vertex set is closed and discrete in this metric, so properness
implies that each bounded ball contains only finitely many vertices.
A connected subset of $\wt T$ is consequently unbounded if and only
if it contains infinitely many vertices. We call a locally finite
weighted tree a \emph{weighted skeleton} of a metric tree if its cable
system is isometric to that metric tree.

Let $m$ be length measure on $\wt T$, and let $\mathcal D_0$ consist
of the continuous compactly supported functions whose restriction to
each cable is absolutely continuous with square-integrable derivative.
The Brownian diffusion used here is the
minimal diffusion associated with the closure of the following form
on $L^2(m)$:
\begin{equation}
 \mathcal E(f,f)=\frac12\int_{\wt T}|f'(z)|^2\,m(dz),
 \qquad f\in\mathcal D_0.
 \label{eq:brownian-energy}
\end{equation}
No killing refers to the absence of a killing term in the form;
we do not assume that the minimal diffusion has infinite lifetime.
Write $k_t(u,v)$ for its transition density with respect to $m$. Its
Green function is
\begin{equation}
 g(u,v)=\int_0^\infty k_t(u,v)\,dt,
 \qquad u,v\in\wt T.
 \label{eq:green}
\end{equation}
We use the electrical convention that an edge of conductance $c_e$
has resistance $c_e^{-1}$. The effective resistance from a vertex
to infinity is the limit of the resistances to the complements of
an increasing finite exhaustion, with each complement identified
as a single boundary vertex. Effective conductance is its reciprocal.
At vertices, $g$ is the electrical Green function of $(T,c)$:
$g(x,y)$ is the voltage at $y$ when one unit of current enters at
$x$ and exits at infinity. In particular, $g(x,x)$ is the effective
resistance from $x$ to infinity.

For $z\in\wt T$ and $t>0$, let $\mathsf P^t_{z,z}$ be the Brownian
bridge probability measure of duration $t$ from $z$ to $z$. We regard
loops up to circular shifts of their time parameter, and also write
$\mathsf P^t_{z,z}$ for the image of the bridge measure on this loop
space. The Brownian loop measure is
\begin{equation}
 \wt\nu(\,\cdot\,)
 =\int_{\wt T}\int_0^\infty
       \frac{k_t(z,z)}{t}\,\mathsf P^t_{z,z}(\,\cdot\,)\,dt\,m(dz),
 \label{eq:loop-measure}
\end{equation}
in the normalization of~\cite[Section~2]{Lupu16}. Under
$\Pp_\alpha^T$, the soup is a Poisson point process of loops with
intensity $\alpha\wt\nu$, and we denote its expectation by $\E_\alpha^T$.
The trace of a loop is the set of cable points it visits.
Two loops belong to the same cluster if they can be joined by a
finite chain of loops whose successive traces intersect. A cluster
is the union of the traces in one such equivalence class. We write
$u\leftrightarrow v$ when $u$ and $v$ belong to the same cluster,
and $u\leftrightarrow\infty$ when the cluster containing $u$ is
unbounded. For $\alpha>0$, set
\begin{equation}
 p_\alpha(u,v)=\Pp_\alpha^T(u\leftrightarrow v),\qquad
 p_\alpha(u,u)=1,
 \label{eq:connection-probability}
\end{equation}
and define the critical intensity by
\begin{equation}
 \ac(T)=\inf\{\alpha>0:
    \Pp_\alpha^T(\text{there is an unbounded cluster})>0\}.
 \label{eq:critical}
\end{equation}
Our supercritical assertions give positive probability of
$u\leftrightarrow\infty$ for every fixed vertex $u$. We reserve
$\Pp_\alpha^T$ and $\E_\alpha^T$ for the soup, unindexed $\Pp$ and
$\E$ refer to auxiliary random variables with the specified laws.

Fix a root $o\in V$, for $x\in V$, let $|x|$ be its graph
distance from $o$, and let $T_x$ be the subtree consisting of $x$
and its descendants. For $x\ne o$, write $x^-$ for its parent.
We also write $\de(x)$ for the degree
of $x$, $[u,v]$ for the unique geodesic between two cable points, and
$x\wedge y$ for the last common ancestor of two vertices. The
boundary $\partial T$ is the set of infinite rays
$\xi=(\xi_0,\xi_1,\ldots)$ with $\xi_0=o$ and $|\xi_n|=n$.
For $x\in V$, the \emph{cylinder}
$\partial T_x=\{\xi\in\partial T:\xi_{|x|}=x\}$ consists of the
rays passing through $x$. These sets generate the topology and
sigma-field on $\partial T$. For distinct rays $\xi,\eta$, write
$\xi\wedge\eta$ for their last common vertex.
Define the normalized Green correlations by
\begin{equation}
 \rho(u,v)=\frac{g(u,v)}{\sqrt{g(u,u)g(v,v)}},\qquad
 \rho_o(x)=\rho(o,x).
 \label{eq:rho}
\end{equation}
A \emph{cutset} is a finite set of vertices meeting every ray in
$\partial T$, no two distinct members of which are ancestors of one
another. Write $\mathcal C_o$ for the set of these cutsets; in
particular, $\{o\}\in\mathcal C_o$. The \emph{branching number} is
\begin{equation}
 \operatorname{br}(T)=\sup\left\{\lambda\ge1:
   \inf_{\Pi\in\mathcal C_o}\sum_{x\in\Pi}\lambda^{-|x|}>0\right\}.
 \label{eq:branching-number}
\end{equation}
For independent bond percolation, in which each edge is retained
independently with probability $p$, the critical probability is
$1/\operatorname{br}(T)$~\cite{Lyons90}. The weight $p^{|x|}$ is
the probability that the root is connected to $x$. For the Brownian
loop soup, the corresponding connection probabilities are comparable
to powers of $\rho_o(x)$; see \eqref{eq:power}. We therefore define
\begin{equation}
 D_G(T,o)=\sup\left\{s\ge0:
       \inf_{\Pi\in\mathcal C_o}\sum_{x\in\Pi}\rho_o(x)^s>0\right\}.
 \label{eq:DG}
\end{equation}
The relation with \eqref{eq:branching-number} is geometric:
Green correlations multiply along geodesics, as stated in
\eqref{eq:multiplicative}, and hence the function
\begin{equation}
 d_G(u,v)=-\log\rho(u,v),\qquad
 d_G(o,x)=\sum_{e=uv\in[o,x]}-\log\rho(u,v),
 \label{eq:green-metric}
\end{equation}
is a path metric with positive edge lengths. On the boundary, put
\begin{equation}
 \delta_G(\xi,\eta)=
 \begin{cases}
  e^{-d_G(o,\xi\wedge\eta)}=\rho_o(\xi\wedge\eta),&\xi\ne\eta,\\
  0,&\xi=\eta.
 \end{cases}
 \label{eq:green-boundary-metric}
\end{equation}
This is an \emph{ultrametric}: it satisfies the stronger triangle
inequality $\delta_G(\xi,\zeta)\le
\max\{\delta_G(\xi,\eta),\delta_G(\eta,\zeta)\}$.
We write $\dim_{\mathrm H}$ for Hausdorff dimension, the critical
exponent for sums of powers of diameters in covers by arbitrarily
small sets. The classical cutset argument~\cite[Section~1.8]{LP} gives
\begin{equation}
 D_G(T,o)=\dim_{\mathrm H}(\partial T,\delta_G).
 \label{eq:green-boundary-dimension}
\end{equation}
We prove compactness of this boundary metric and
\eqref{eq:green-boundary-dimension} in Lemma~\ref{lem:frostman}.
In comparison, $\log\operatorname{br}(T)$ is the boundary dimension
for the ultrametric $e^{-|\xi\wedge\eta|}$, with zero distance on
the diagonal~\cite[Chapters~1 and~15]{LP}. Thus $D_G$ uses
Green distance in place of graph distance in the same cutset
construction. Figure~\ref{fig:green-boundary} illustrates the metric,
cylinders and cutsets.

\begin{figure}[htbp]
 \centering
 \input{figures/cable_green_boundary.tex}
 \caption{Two rays and their Green boundary distance. The shared path
 $[o,a]$ is bold, and the shaded subtree has boundary $\partial T_x$.
 Encircled vertices form a cutset $\Pi$ at different graph depths.
 Dashed ends indicate continuation.}
 \label{fig:green-boundary}
\end{figure}
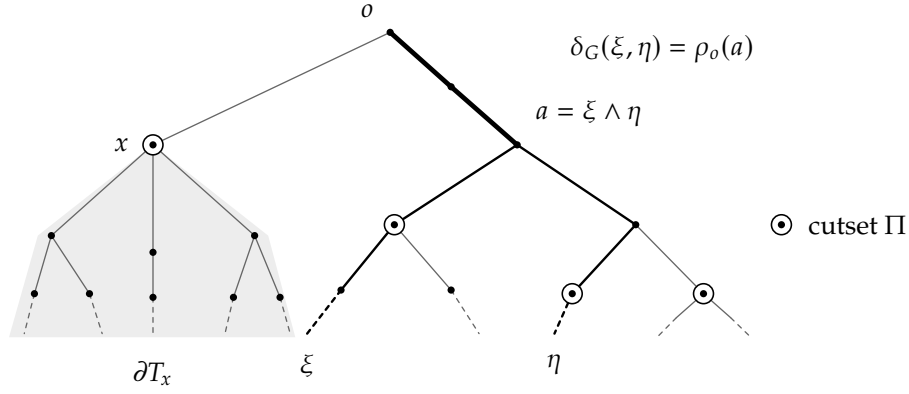

Unless a different tree or root is specified, all cutset infima
refer to $\mathcal C_o$. For positive
functions, $f\asymp g$ means that their ratio is bounded above and
below by positive constants, uniformly over the stated variables.

Our main result expresses the critical intensity in terms of this
boundary dimension.
\begin{theorem}
\label{thm:classification}
Under \eqref{eq:weighted-setting}, the exponent $D_G(T,o)$ is
independent of $o$, takes values in $[0,2]$, and satisfies
\begin{equation}
 \ac(T)=1-\frac{D_G(T,o)}2
       =1-\frac12\dim_{\mathrm H}(\partial T,\delta_G).
 \label{eq:classification}
\end{equation}
The exponent is also independent of the choice of a locally finite
weighted skeleton representing the same cable tree. At intensity one,
the whole cable tree is almost surely a single cluster.
\end{theorem}

We consequently write $D_G(T)$, or $D_G$ when the tree is fixed. Note that the theorem identifies the critical
\emph{value}, but does not generally settle percolation at that value
when $0<D_G(T)<2$. If $D_G=2$, percolation occurs with positive
probability at every positive intensity.

Green metrics and their boundary dimensions have an established
history. Blach\`ere and Brofferio~\cite{BB} introduced the hitting
distance $-\log H(u,v)$ for symmetric random walks on groups, where
$H(u,v)$ is the probability of ever visiting $v$ when started at $u$.
Blach\`ere, Ha\"issinsky and Mathieu~\cite[Theorem~1.1]{BHM}
study boundary dimension and harmonic measure for the resulting
Green metrics on hyperbolic groups. In our reversible setting,
\eqref{eq:hitting} identifies $d_G(u,v)$ with
$-\tfrac12\log[H(u,v)H(v,u)]$, the symmetrization of the hitting
distance.

There is also a direct relation to the Na\"im kernel,
which describes the energy of harmonic functions through their
boundary values. In our Green-function normalization it is, for
distinct rays,
\begin{equation}
 \Theta_o(\xi,\eta)
 =\lim_{\substack{x\to\xi\\y\to\eta}}
       \frac{g(x,y)}{g(o,x)g(o,y)}
 =\frac1{g(o,o)\delta_G(\xi,\eta)^2},
 \label{eq:green-naim}
\end{equation}
where $x$ and $y$ tend to infinity along the indicated rays.
This is the tree kernel in~\cite[equation~(6.17)]{BGWP}; the
boundary ultrametric construction is treated in its Section~6.1.
In particular, the metric in \eqref{eq:green-boundary-metric} is
the normalized inverse square root of the Na\"im kernel.

\subsection{The threshold one-half: capacity and finite energy}

Let $\varphi=(\varphi_z)_{z\in\wt T}$ be the cable Gaussian free
field, that is, the continuous centered Gaussian field with
covariance $g$, and denote its law and expectation
by $\Pp^{\mathrm G,T}$ and $\E^{\mathrm G,T}$. Its sign clusters are
the connected components of $\{\varphi\ne0\}$. Lupu's coupling
identifies these with the loop clusters at intensity $1/2$, and his
two-point formula~\cite[Proposition~5.2]{Lupu16} gives
\begin{equation}
 p_{1/2}(u,v)
 =\E^{\mathrm G,T}[\operatorname{sgn}(\varphi_u)
                         \operatorname{sgn}(\varphi_v)]
 =\frac2\pi\arcsin\rho(u,v).
 \label{eq:lupu-two-point}
\end{equation}
For trees, we obtain an exact extension to every $0<\alpha<1$ in
\eqref{eq:beta}. In particular, \eqref{eq:power} shows that the
connection probability is comparable to $\rho(u,v)^{2(1-\alpha)}$,
with constants depending only on $\alpha$. This exponent accounts
for the factor $1/2$ in \eqref{eq:classification}.

For a compact set $A\subseteq\wt T$, define its Brownian capacity by
\begin{equation}
 \capacity(A)=\inf\{\mathcal E(f,f):f\in\mathcal D_0,\ f\ge1
                         \text{ on }A\}.
 \label{eq:capacity-definition}
\end{equation}
For an arbitrary set, capacity is the supremum over its compact
subsets. For a finite vertex set it equals the effective conductance
to infinity after \emph{wiring} that set, that is, identifying its
vertices as a single vertex. We consider the capacity assumption
\begin{equation}
 \capacity(A)=\infty\quad\text{for every closed connected
 unbounded }A\subseteq\wt T.
 \label{eq:capacity-condition}
\end{equation}
Under properness, this is equivalent to the vertex condition
\begin{equation}
 \capacity(B)=\infty\quad\text{for every infinite connected }B\subseteq V.
 \label{eq:vertex-capacity-condition}
\end{equation}
Indeed, every connected unbounded cable set contains such a vertex
set. Conversely, the union of the cables joining vertices of an
infinite connected vertex set is closed and unbounded, and has the
same capacity by linear interpolation in \eqref{eq:capacity-definition}.
This is the vertex formulation of condition (Cap) in Drewitz,
Pr\'evost and Rodriguez~\cite[Lemma~3.4(1)]{DPR}.
Their Remark~3.8(1) permits omission of
the auxiliary half-lines attached at vertices when killing vanishes,
so their results apply to our cable convention.
They prove that Gaussian sign clusters have finite Brownian capacity
\cite[Theorem~3.2]{DPR}.
Together with Lupu's coupling, \eqref{eq:capacity-condition}
therefore gives non-percolation at $1/2$.

Chang, Du and Li~\cite{CDL} prove
percolation at every intensity above $1/2$ on transient
quasi-transitive metric graphs, whose conductance-preserving
automorphisms have finitely many vertex orbits. Their argument adds independent
local connections to the soup at intensity $1/2$. An exploration
from a distant boundary, together with the killed version of
\eqref{eq:lupu-two-point}, yields a differential inequality for the
connection probability. Uniform local bounds make the inequality
independent of the distance to the boundary. Any additional positive
soup intensity supplies the required connections, giving
percolation. See~\cite[Proposition~7]{CDL} and the sharpness method
of Duminil-Copin and Tassion~\cite{DCT} on which it builds.

To formulate the upper-bound criterion without degree assumptions,
we introduce flows on $T$.
A \emph{unit flow} is a probability measure $\mu$ on $\partial T$;
we write $\mu(x)=\mu(\partial T_x)$. Then
\begin{equation}
 \mu(o)=1,\qquad
 \mu(x)=\sum_{y:y^-=x}\mu(y),\qquad
 \sum_{x\in\Pi}\mu(x)=1\quad(\Pi\in\mathcal C_o).
 \label{eq:flow-conservation}
\end{equation}
For $x\ne o$, the value $\mu(x)$ is the current assigned to the edge $x^-x$.
In a network with resistance $r(x)>0$ on that edge, its electrical
energy is $\sum_{x\ne o}r(x)\mu(x)^2$. The network is transient
if and only if some unit flow has finite energy
\cite[Theorem~2.11]{LP}.

For $x\in V$, let $C_x^+$ be the effective conductance from $x$ to
infinity through $T_x$. It is zero if $T_x$ is finite or recurrent.
The \emph{electrical unit flow} $\theta$ is the unit flow minimizing
energy for the original resistances $c_{x^-x}^{-1}$. It is given by
\begin{equation}
 \theta(o)=1,\qquad
 \theta(x)=c_{x^-x}\bigl(g(o,x^-)-g(o,x)\bigr)=C_x^+g(o,x)
 \quad(x\ne o).
 \label{eq:electrical-flow}
\end{equation}
It defines a probability measure on $\partial T$, called harmonic
measure: $\theta(x)$ is the measure of rays passing through $x$.
For $x\ne o$ with $\theta(x)>0$, put
\begin{equation}
 \Lambda(x)=(C_x^+)^2g(x,x),\qquad
 \rho_o(x)^2=\frac{\theta(x)^2}{g(o,o)\Lambda(x)}.
 \label{eq:current-factor}
\end{equation}
We write $\log^+t=\max\{0,\log t\}$ for $t>0$.

\begin{theorem}
\label{thm:half-analytic}
Suppose \eqref{eq:weighted-setting} and \eqref{eq:capacity-condition}
hold. Then $D_G(T)\le1$, $\ac(T)\ge1/2$, and all clusters at
intensity $1/2$ are bounded almost surely. Moreover,
\begin{equation}
 \ac(T)=\frac12\quad\Longleftrightarrow\quad
 \inf_{\Pi\in\mathcal C_o}\sum_{x\in\Pi}\rho_o(x)^s>0
       \quad\text{for every }0<s<1.
 \label{eq:half-condition}
\end{equation}
Equivalently, for every $0<\beta<1$, the network with edge
resistances
\begin{equation}
 r_\beta(x)=\rho_o(x)^{-\beta}-\rho_o(x^-)^{-\beta}
          \quad(x\ne o)
 \label{eq:analytic-resistances}
\end{equation}
is transient. A sufficient condition for these equivalent conclusions
is that a set of rays of positive harmonic measure satisfies
\begin{equation}
 \frac{\log^+\Lambda(\xi_n)}{\log(1/\theta(\xi_n))}
                  \longrightarrow0.
 \label{eq:analytic-ray-condition}
\end{equation}
\end{theorem}

The resistances in \eqref{eq:analytic-resistances} are positive,
since $\rho_o$ decreases strictly away from $o$; see
\eqref{eq:multiplicative}. Their transience for every $0<\beta<1$
is equivalent to $D_G\ge1$ even without
\eqref{eq:capacity-condition}, and hence characterizes the upper
bound $\ac\le1/2$. By \eqref{eq:green-boundary-dimension}, the
boundary must have Green-metric Hausdorff dimension at least one;
under \eqref{eq:capacity-condition}, it then has dimension exactly
one. The corresponding finite-energy flows may
depend on $\beta$. Condition~\eqref{eq:analytic-ray-condition}
gives a sufficient way to construct such flows using harmonic measure.
Its denominator
tends to infinity along every ray of positive current by
\eqref{eq:harmonic-nonatomic}.

For unit conductances, \eqref{eq:capacity-condition} has the exact
ray test \eqref{eq:ray-capacity-condition}: along every infinite ray,
sum the conductances from its vertices to infinity through branches
disjoint from the rest of the ray. This sum must diverge. Also
$\Lambda(x)\le C_x^+\le\de(x)-1$ for $x\ne o$, so the degree condition
\eqref{eq:ray-assumption} implies \eqref{eq:analytic-ray-condition}.
Neither assertion requires a lower degree bound.

A useful special case is
\begin{equation}
 c_e=1\quad(e\in E),\qquad \de(x)\ge3\quad(x\in V).
 \label{eq:standing}
\end{equation}
These assumptions imply \eqref{eq:weighted-setting} and
\eqref{eq:capacity-condition}, and hence the one-half lower bound;
see \eqref{eq:unit-lower-bound}.
When the degrees are also bounded above, percolation above one-half
follows from Chang--Du--Li~\cite[Remark~9(1)]{CDL}. On the
$d$-regular tree with unit conductances, for an integer $d\ge3$,
we have $\operatorname{br}(T)=d-1$ and
$\rho_o(x)=(d-1)^{-|x|}$. Green distance is therefore graph distance
multiplied by $\log(d-1)$, so the boundary dimension in
\eqref{eq:green-boundary-dimension} is one for every $d$.
The same correlation formula and the generation size
$d(d-1)^{n-1}$ for $n\ge1$ give
\begin{equation}
 \E_\alpha^T\#\{x\in V:|x|=n,\ o\leftrightarrow x\}
 \asymp(d-1)^{(2\alpha-1)n},\qquad 0<\alpha<1,
 \label{eq:regular-first-moment}
\end{equation}
by \eqref{eq:power}. The change at one-half in
\eqref{eq:regular-first-moment} reflects the competition between
the number of vertices and the decay of connection probabilities.
The dependence between connections must also be controlled; this
is the purpose of \eqref{eq:quasi}. For general trees, generations
alone need not identify the threshold, and \eqref{eq:DG} uses all
cutsets.

The discrete loop soup has a different critical intensity even on
regular trees. In our work with Betz, Drewitz and Klippel~\cite{BDKM},
we prove that its threshold on the $d$-regular tree, $d\ge3$, is
strictly larger than the cable threshold $1/2$. Makowiec and
Sapozhnikov~\cite{MS} determine an implicit critical curve and
critical behaviour for the discrete model on regular trees. Our
classification concerns the cable model and allows both the degrees
and the conductances to vary with location.

\subsection{Explicit thresholds and comparison with Gaussian levels}

The lower bound $1/2$ under \eqref{eq:standing} need not be an
equality when the degrees are unbounded. The next consequence of
\eqref{eq:classification} describes this dependence explicitly.
We call a rooted weighted tree \emph{spherically symmetric} if the
number of children depends only on the depth of the parent and an
edge's conductance depends only on the depth of its endpoints.

\begin{corollary}
\label{cor:symmetric}
Suppose all conductances are one and every vertex at depth $n$ has
$b_n$ children, where $b_0\ge3$ and $b_n\ge2$ for $n\ge1$. Define
\begin{equation}
 N_0=1,\qquad N_{n+1}=N_nb_n,\qquad
 r=\limsup_{n\to\infty}\frac{\log b_n}{\log N_n}\in[0,\infty].
 \label{eq:symmetric-parameters}
\end{equation}
Then
\begin{equation}
 D_G(T)=\frac2{2+r},\qquad
 \ac(T)=\frac{1+r}{2+r},
 \label{eq:symmetric}
\end{equation}
with $D_G(T)=0$ and $\ac(T)=1$ when $r=\infty$.
\end{corollary}

Here $N_n$ is the number of vertices at depth $n$. Thus equality
$\ac(T)=1/2$ holds precisely when
$\log b_n=o(\log N_n)$: the relevant comparison is between the
number of children of one vertex and the size of its generation.
By contrast, the ordinary branching number in this setting is
$\liminf_n N_n^{1/n}$. Sparse generations with many children can
change $r$ while leaving this lower exponential growth rate fixed.
This gives the following consequence.

\begin{corollary}
\label{cor:fixed-branching}
For every $a\in[1/2,1]$, there is a spherically symmetric tree
with unit conductances and minimum degree three such that
\begin{equation}
 \operatorname{br}(T)=2,\qquad \ac(T)=a.
 \label{eq:fixed-branching}
\end{equation}
\end{corollary}

All these trees have Bernoulli critical probability $1/2$, while
their Green boundary dimensions range over $[0,1]$.

The quantitative bound \eqref{eq:degree-conclusion} and the
criterion \eqref{eq:ray-assumption} give further consequences for
trees without spherical symmetry. In particular,
\eqref{eq:subexponential-degrees} gives a sufficient condition in
terms of maximal degrees in each generation.

Varying the conductances permits thresholds below $1/2$, even with
bounded degrees. Pr\'evost~\cite[Section~7]{Prevost} uses regular
trees with geometric edge weights to construct unbounded Gaussian
sign clusters. Chang--Du--Li~\cite[Remark~9]{CDL} ask whether the
loop-soup threshold can be strictly below $1/2$ on such examples.
We compute the threshold for this family throughout its transient
range with proper cable metric and no root killing.

\begin{proposition}
\label{prop:weighted-example}
Fix an integer $b\ge2$. Give the root $b+1$ children and every
other vertex $b$ children. For $1/b<a\le1$, put conductance $a^n$
on each edge from generation $n$ to generation $n+1$. Then
\eqref{eq:weighted-setting} holds and
\begin{equation}
 D_G(T)=\frac{2\log b}{\log(b^2a)},\qquad
 \ac(T)=\frac{\log(ba)}{\log(b^2a)}.
 \label{eq:weighted-example}
\end{equation}
\end{proposition}

For each fixed $b$, \eqref{eq:weighted-example} realizes every
critical intensity in $(0,1/2]$. For example, $b=2$ and
$a=2^{-1/2}$ give $D_G(T)=4/3$ and $\ac(T)=1/3$.
There are also unit-conductance examples below one-half.
Pr\'evost~\cite[Remark~7.2(2)]{Prevost} uses subdivision of weighted
cables to obtain unit conductances and zero killing. To compute an
explicit loop-soup threshold, take $b=4$ and $a=1/2$, and replace
each generation-$n$ edge by $2^n$ unit
edges in series. The cable tree is unchanged. The resulting tree
has degrees two and five, and
\begin{equation}
 D_G(T)=\frac43,\qquad \ac(T)=\frac13.
 \label{eq:subdivided-example}
\end{equation}
Here \eqref{eq:analytic-ray-condition} holds, since the degrees
are bounded, but the off-ray conductances along each ray are
summable and \eqref{eq:capacity-condition} fails. Conversely, the
unit-conductance example with threshold one in Corollary~\ref{cor:fixed-branching}
satisfies \eqref{eq:capacity-condition} but fails
\eqref{eq:half-condition}. The capacity and flow assumptions thus
serve separate purposes.

The endpoint zero can also be attained with unit conductances and
bounded degrees.
\begin{proposition}
\label{prop:zero-threshold}
Start from the rooted binary tree and replace every edge between
generations $n$ and $n+1$ by
\begin{equation}
 L_n=\left\lceil\frac{2^n}{(n+1)^2}\right\rceil,
 \qquad n\ge0,
 \label{eq:zero-threshold-lengths}
\end{equation}
unit-conductance edges in series. The resulting tree $T$ has maximum
degree three, satisfies \eqref{eq:weighted-setting}, and has
\begin{equation}
 D_G(T)=2,\qquad \ac(T)=0.
 \label{eq:zero-threshold}
\end{equation}
\end{proposition}
In particular, bounded degrees and unit conductances alone give no
positive universal lower bound on the critical intensity.

There is also a distinction between varying the soup intensity and
varying the level of a fixed Gaussian field. For the discrete
Gaussian free field, Sznitman~\cite{Sznitman} proves positivity of
the critical level on regular trees. Ab\"acherli and
Sznitman~\cite{AS} give positivity criteria on transient trees, and
Drewitz, Gallo and Pr\'evost~\cite{DGP} establish positivity on
Galton--Watson trees with random conductances. For the cable field,
define
\begin{equation}
 \wt h_*(T)=\inf\left\{h\in\mathbb R:
   \Pp^{\mathrm G,T}\bigl(\{z\in\wt T:\varphi_z\ge h\}
       \text{ has an unbounded component}\bigr)=0\right\}.
 \label{eq:gff-critical}
\end{equation}

\begin{corollary}
\label{cor:gff-separation}
Under \eqref{eq:weighted-setting} and \eqref{eq:capacity-condition},
$\wt h_*(T)=0$.
\end{corollary}

The Gaussian assertion follows from~\cite[Theorem~1.1(2)]{DPR};
Corollary~\ref{cor:fixed-branching} gives every loop-soup threshold
in $[1/2,1]$ even at branching number two. Thus a cable Gaussian critical level of zero
does not characterize a loop-soup threshold of $1/2$. This answers
negatively the final question in Chang--Du--Li~\cite[Remark~9]{CDL}.
The coupling at intensity $1/2$ does not identify changes in
intensity with changes in the Gaussian level.

\subsection{Galton--Watson trees}

The examples in Corollary~\ref{cor:fixed-branching} show that even
a fixed branching number and a minimum degree assumption do not
determine the threshold. In contrast,
independent offspring with the same distribution give threshold
$1/2$, even for arbitrarily heavy tails.

Let $\nu$ be a probability distribution on $\{0,1,2,\ldots\}$ and
let $Z$ have law $\nu$, with mean in $(1,\infty]$. An ordinary
Galton--Watson tree with offspring law $\nu$ gives every vertex an
independent number of children with this law. We allow a different
root law: give $o$ an almost surely finite number $K\ge0$ of children
and attach independent ordinary Galton--Watson trees at these
children. All conductances are one.
The root variable $K$ is independent of these trees and may be
deterministic or random, with positive probability of being nonzero.
This includes the ordinary convention $K\sim\nu$ and the augmented convention
$K=Z_0+1$ with $Z_0\sim\nu$ independent of the other offspring.
Denote the tree law by $\Pt$ and its expectation by $\Et$; we also
use these symbols for probabilities and expectations of the generic
offspring variable $Z$. We write $\nu(k)=\nu(\{k\})$ and
$\mathsf S=\{|T|=\infty\}$ for the survival event.

\begin{theorem}
\label{thm:GW}
For $\Pt(\,\cdot\mid\mathsf S)$-almost every $T$, the tree is
transient and
\begin{equation}
 D_G(T)=1,\qquad \ac(T)=\frac12.
 \label{eq:GW}
\end{equation}
On each such tree there is no percolation at intensity $1/2$, and
the root has positive probability of percolating at every
$\alpha>1/2$.
\end{theorem}

Earlier work already gives information at intensity $1/2$ and,
for bounded offspring, determines the critical intensity.
Ab\"acherli and Sznitman~\cite[Propositions~2.2 and~5.2]{AS}
give a criterion for bounded cable Gaussian sign clusters and
verify it almost surely for supercritical Galton--Watson trees
with finite mean offspring, conditioned on non-extinction.
Through Lupu's coupling~\cite[Proposition~2.1]{Lupu16}, this implies
non-percolation of the Brownian loop soup at intensity $1/2$ and
hence a lower bound of $1/2$ for its critical intensity.
When $\nu$ has bounded support, Chang, Du and
Li~\cite[Remark~9(1)]{CDL} give percolation at every
$\alpha>1/2$, since each realization has bounded degrees and unit
conductances. Combined with the lower bound, this already yields
$\ac(T)=1/2$ in that case. To our knowledge, \eqref{eq:GW} is the
first determination of the critical intensity for arbitrary
supercritical offspring laws with finite offspring almost surely,
without any moment assumption.

In particular, the result allows both infinite mean and infinite
logarithmic moment. The law in \eqref{eq:heavy-tail-example} has
infinite logarithmic moment and no finite positive power moment.

\subsection{An intrinsic formulation for real trees}

A \emph{real tree} is a metric space in which any two points are
joined by a unique simple arc, and this arc is isometric to an
interval. Let $(\mathcal T,o)$ be a rooted real tree. It is \emph{spanned by
its infinite rays} if every point belongs to an isometric image of
$[0,\infty)$ starting at $o$. A \emph{point cutset} is a finite set
meeting every such ray, with no distinct points $x,y$ satisfying
$x\in[o,y]$. For Brownian motion on $\mathcal T$ with length
measure and the energy normalization \eqref{eq:brownian-energy},
define $D_G(\mathcal T,o)$ by \eqref{eq:rho} and \eqref{eq:DG},
using its Green function and point cutsets. Define
$\ac(\mathcal T)$ by \eqref{eq:critical} for its Brownian loop soup.

\begin{corollary}
\label{cor:real-tree}
Suppose $\mathcal T$ is proper, is spanned by its infinite rays,
and its Brownian motion with length measure is transient, with no
killing. Then $D_G(\mathcal T,o)$ is independent of the root and
\begin{equation}
 0\le D_G(\mathcal T,o)\le2,\qquad
 \ac(\mathcal T)=1-\frac{D_G(\mathcal T,o)}2.
 \label{eq:real-tree-classification}
\end{equation}
At intensity one the whole tree is almost surely a single cluster.
\end{corollary}

Under these assumptions, the skeleton result of
Athreya--Eckhoff--Winter~\cite[Lemma~7.1 and Remark~7.2]{AEW}
reduces the real-tree statement to \eqref{eq:classification}.
The equality of the vertex and point cutset infima is recorded in
\eqref{eq:point-cutsets}.

All proofs are collected in Section~\ref{sec:proofs}.
Subsection~\ref{sec:electrical} proves the Green-function identities.
Subsections~\ref{sec:occupation} and~\ref{sec:cutsets} establish the
connection estimates and prove Theorem~\ref{thm:classification}.
Subsection~\ref{sec:metric} proves Corollary~\ref{cor:real-tree} and
Propositions~\ref{prop:weighted-example} and~\ref{prop:zero-threshold}.
Subsection~\ref{sec:unit} proves Theorem~\ref{thm:half-analytic},
the pruning and ray-capacity criteria, and
Corollaries~\ref{cor:symmetric}, \ref{cor:fixed-branching}
and~\ref{cor:gff-separation}.
Subsection~\ref{sec:GW-proof} proves Theorem~\ref{thm:GW}.

\section{Proofs}
\label{sec:proofs}

The probabilistic input to the classification consists of the
two-point formula \eqref{eq:beta} and the branching-point comparison
\eqref{eq:quasi}. To obtain them, we use the occupation field, which
records the total local time of the loops at each point. After
normalization by the diagonal Green function, its restriction to a
geodesic is a diffusion related to a squared-Bessel process.
The interval description is classical~\cite{Lupu18}, and
Ray--Knight identities at general intensities on trees are studied
by Li and Zheng~\cite{LZ}. We give the finite-dimensional calculation
that identifies the normalized field and its conditional independence
across branches. The squared-Bessel hitting law then gives
\eqref{eq:beta}; concavity of the conditional connection probabilities
gives \eqref{eq:quasi}, with a constant depending only on the intensity.

These estimates allow us to apply the classical boundary-capacity
method on trees. The two-point function is comparable to a power
of the normalized Green function, so first and second moments give
cutset and finite-energy criteria in the metric
\eqref{eq:green-boundary-metric}. Lemma~\ref{lem:frostman} identifies
the cutset exponent with Hausdorff dimension and constructs a
finite-energy flow below this dimension by max-flow/min-cut.
We include the energy estimate because it
applies without uniform degree or conductance bounds. Finally,
\eqref{eq:green-resistance} relates the normalized Green function
to a recurrent network and gives the universal bound $D_G\le2$.

\subsection{Green distance and electrical resistance}
\label{sec:electrical}

We begin with identities relating the Green function to hitting
probabilities and electrical resistance. The recurrence statement in
Lemma~\ref{lem:green-transform} will give the bound $D_G\le2$ through
\eqref{eq:energy-two}.

For the hitting probabilities $H(u,v)$ introduced above, the Green
hitting identity and symmetry give
\begin{equation}
 g(u,v)=H(u,v)g(v,v),\qquad
 \rho(u,v)^2=H(u,v)H(v,u).
 \label{eq:hitting}
\end{equation}
If $v$ lies on the path from $u$ to $w$, all trajectories from $u$ to
$w$ hit $v$, so
\begin{equation}
 \rho(u,w)=\rho(u,v)\rho(v,w).
 \label{eq:multiplicative}
\end{equation}
The same identity holds for cable points. For $u\ne v$, the Green
matrix on $\{u,v\}$ is strictly positive definite, so
$0<\rho(u,v)<1$. Consequently $\rho_o$ decreases strictly along
each geodesic away from $o$.

For distinct rays $\xi,\eta$, set $v=\xi\wedge\eta$ and take
$x\in\xi$, $y\in\eta$ beyond $v$. Applying
\eqref{eq:multiplicative} to the three Green functions in the ratio
gives
$g(x,y)/(g(o,x)g(o,y))=g(v,v)/g(o,v)^2$.
The ratio is already constant once $x,y$ lie beyond $v$, proving
\eqref{eq:green-naim}.

Recall that $C_x^+$ is the effective conductance from $x$ to
infinity through $T_x$, which may vanish.

\Needspace{12\baselineskip}
\begin{lemma}
\label{lem:green-transform}
Set $g_o=g(o,o)$ and
\begin{equation}
 h(x)=\frac{g(o,x)}{g_o}=H(x,o),\qquad
 \widehat c_{xy}=c_{xy}h(x)h(y).
 \label{eq:green-transform}
\end{equation}
The network $(T,\widehat c)$ is recurrent. For every vertex $x$,
\begin{equation}
 \rho_o(x)^{-2}
 =1+\frac1{g_o}\sum_{e\in[o,x]}\frac1{\widehat c_e}.
 \label{eq:green-resistance}
\end{equation}
In particular,
\begin{equation}
 \rho_o(\xi_n)\downarrow0\qquad\text{for every }\xi\in\partial T.
 \label{eq:ray-correlation-decay}
\end{equation}
\end{lemma}

\begin{proof}
The function $h$ is positive, equals one at $o$, and is harmonic
off $o$. For $x\ne o$, the transition probabilities of the
transformed network satisfy
\begin{equation}
 \widehat P(x,y)=P(x,y)\frac{h(y)}{h(x)}\qquad(y\sim x).
 \label{eq:transformed-transition}
\end{equation}
Up to its first visit to $o$, this is the original walk conditioned
to hit $o$. Every departure from $o$ therefore returns almost surely,
which proves recurrence.

Fix an edge from $p=x^-$ to $x$. The voltage
formula and the Green function killed on hitting $p$ give
\begin{equation}
 H(x,p)=\frac{c_{px}}{c_{px}+C_x^+}=\frac{h(x)}{h(p)},\qquad
 g(x,x)-\frac{g(p,x)^2}{g(p,p)}=\frac1{c_{px}+C_x^+}.
 \label{eq:edge-hitting-green}
\end{equation}
Both identities in \eqref{eq:edge-hitting-green} follow by finite
network exhaustion, including when $C_x^+=0$. Since
$g(p,x)=H(x,p)g(p,p)$, division yields
\begin{equation}
 \frac{g(x,x)}{h(x)^2}-\frac{g(p,p)}{h(p)^2}
 =\frac1{c_{px}h(p)h(x)}=\frac1{\widehat c_{px}}.
 \label{eq:green-increment}
\end{equation}
Summing \eqref{eq:green-increment} along $[o,x]$ and using
$g(x,x)/h(x)^2=g_o\rho_o(x)^{-2}$ proves
\eqref{eq:green-resistance}.
Every infinite ray has infinite total resistance in a recurrent
network: otherwise the unit flow supported on that ray would have
finite energy, contradicting the electrical transience criterion
\cite[Theorem~2.11]{LP}.
Equation~\eqref{eq:green-resistance} therefore gives
\eqref{eq:ray-correlation-decay}.
\end{proof}

\subsection{Occupation fields and connection probabilities}
\label{sec:occupation}

We first describe the occupation field along geodesics and its
conditional independence across branches. These properties will give
the connection formula \eqref{eq:beta} and the comparison
\eqref{eq:quasi}.

Let $L_t^z$ be Brownian local time at $z$, normalized with respect to
the length measure $m$: for the Brownian diffusion $W$ and every
nonnegative measurable $f$,
\begin{equation}
 \int_0^t f(W_s)\,ds=\int_{\wt T}f(z)L_t^z\,m(dz)
 \label{eq:local-time-normalization}
\end{equation}
up to its lifetime. If $t_\gamma$ is the duration of a loop $\gamma$
and $\mathcal L_\alpha$ is the Poisson soup, its occupation field is
\begin{equation}
 \ell_\alpha(z)=\sum_{\gamma\in\mathcal L_\alpha}
                       L_{t_\gamma}^z(\gamma),\qquad z\in\wt T.
 \label{eq:occupation}
\end{equation}
By~\cite[Lemmas~2.2 and~2.3]{Lupu16}, this field has a continuous
version whose positive components are precisely the loop clusters.
We use that version throughout. In particular, on a tree,
\begin{equation}
 \{u\leftrightarrow v\}
 =\{\ell_\alpha(z)>0\text{ for all }z\in[u,v]\}
 \quad\text{almost surely}.
 \label{eq:positive-path}
\end{equation}
For distinct cable points $z_1,\ldots,z_k$ and numbers
$\lambda_1,\ldots,\lambda_k\ge0$, set
$G=(g(z_i,z_j))_{1\le i,j\le k}$ and
$D=\operatorname{diag}(\lambda_1,\ldots,\lambda_k)$.
The occupation-field identity~\cite[Chapter~4]{LeJan} gives
\begin{equation}
 \E_\alpha^T\exp\left\{-\sum_{i=1}^k\lambda_i\ell_\alpha(z_i)\right\}
 =\det(I+GD)^{-\alpha},
 \label{eq:laplace}
\end{equation}
where $I$ is the identity matrix of the indicated size. In this
normalization the occupation isomorphism at intensity $1/2$ is
$\ell_{1/2}\overset d=\varphi^2/2$.

For a fixed intensity $\alpha>0$, write
\begin{equation}
 X(z)=\frac{\ell_\alpha(z)}{g(z,z)},\qquad z\in\wt T.
 \label{eq:normalized-occupation}
\end{equation}
Each $X(z)$ has law $\Ga(\alpha,1)$, the Gamma distribution with
density $r^{\alpha-1}e^{-r}/\Gamma(\alpha)$ on $(0,\infty)$.
Let $R=(\rho(z_i,z_j))_{1\le i,j\le k}$. Rescaling the variables in
\eqref{eq:laplace} and conjugating by the diagonal matrix with
entries $\sqrt{g(z_i,z_i)}$ gives
\begin{equation}
 \E_\alpha^T\exp\left\{-\sum_{i=1}^k\lambda_iX(z_i)\right\}
 =\det(I+RD)^{-\alpha}.
 \label{eq:normalized-laplace}
\end{equation}

\paragraph*{The occupation field on a tree.}

For the corresponding occupation-field diffusion on an interval, see
Lupu~\cite[Proposition~4.6 and equation~(4.2.4)]{Lupu18}.
Conditional independence across branches follows from the excursion
decomposition in Li and Zheng~\cite[Proposition~2.2 and Section~3.1.1]{LZ}.
We give a direct proof in our normalization.

\begin{lemma}
\label{lem:branching}
For any cable point $a$, conditional on $X(a)$, the fields in the components of
$\wt T\setminus\{a\}$ are independent. Along a geodesic directed
away from $a$, with parameter $t=-\log\rho(a,\cdot)^2$, the field
has the law of the stationary diffusion, restricted to the range of
this parameter,
\begin{equation}
 dX_t=\sqrt{2X_t}\,dB_t+(\alpha-X_t)\,dt,\qquad
 X_0\sim\Ga(\alpha,1).
 \label{eq:CIR}
\end{equation}
Here $B$ is a standard one-dimensional Brownian motion, independent
of $X_0$.
\end{lemma}

\begin{proof}
For $0<q\le1$, let $Q_q(z,\cdot)$ be the transition law of
\eqref{eq:CIR} after time $-\log q$, started at $z\ge0$, and write
$Q_qf(z)=\int f(y)Q_q(z,dy)$. Its Laplace transform is
\begin{equation}
 Q_q(e^{-\lambda\,\cdot})(z)
 =[1+(1-q)\lambda]^{-\alpha}
   \exp\left\{-\frac{qz\lambda}{1+(1-q)\lambda}\right\},\qquad\lambda\ge0.
 \label{eq:CIR-LT}
\end{equation}
Indeed, if $Y$ solves the stochastic differential equation
\begin{equation}
 dY_s=\sqrt{2Y_s}\,dB_s+\alpha\,ds,
 \label{eq:BESQ}
\end{equation}
driven by a standard Brownian motion, then $2Y$ is a squared-Bessel
process of dimension $2\alpha$ and the time change
\begin{equation}
 X_t=e^{-t}Y_{e^t-1}
 \label{eq:bessel-time-change}
\end{equation}
solves \eqref{eq:CIR} in law, with the same initial value.
Formula~\eqref{eq:CIR-LT} follows from the squared-Bessel
transition transform; see~\cite[Sections~4.1--4.2]{Lupu18}.
Integration against $\Ga(\alpha,1)$ shows that this law is invariant.

Fix finitely many points, including $a$. Their spanning tree is
the union of the geodesics joining these points, and we add its branch
points to the finite set. By \eqref{eq:multiplicative}, the normalized
Green matrix $R$ on this enlarged finite set has entries given by
products of correlations across the edges of its spanning tree.
Starting at $a$ with law $\Ga(\alpha,1)$, use the kernels $Q_q$
independently on the branches, where $q$ is the squared correlation
across an edge. We verify that the resulting field has transform
$\det(I+RD)^{-\alpha}$.

Let $l$ be a leaf with parent $p$, and put $q=R_{pl}^2$ and
$A=1+(1-q)\lambda_l$. Let $R_0$ and $D_0$ be the matrices
with the coordinate $l$ deleted, and let $e_p$ be the coordinate
vector at $p$ in the remaining space. By \eqref{eq:CIR-LT},
integrating out $X(l)$ multiplies the transform by $A^{-\alpha}$
and replaces $\lambda_p$ by $\lambda_p+q\lambda_l/A$. The same
recursion holds for the determinant: subtracting $\sqrt q$ times
row $p$ from row $l$ and eliminating the leaf coordinate gives
\begin{equation}
 \det(I+RD)
 =A\det\left(I+R_0\left[
        D_0+\frac{q\lambda_l}{A}e_pe_p^{\mathsf T}\right]\right).
 \label{eq:leaf}
\end{equation}
Induction in \eqref{eq:leaf} proves the transform identity, and
\eqref{eq:normalized-laplace} identifies the finite-dimensional laws.
Continuity on the cables then identifies the fields on countable
dense subsets, and hence their laws as continuous functions.
The independent construction on the branches proves the conditional
independence assertion.
\end{proof}

\paragraph*{Connections along a geodesic.}

For $b,c>0$ and $q\in[0,1]$, define the beta function and the
regularized incomplete beta function by
\begin{equation}
 B(b,c)=\int_0^1 t^{b-1}(1-t)^{c-1}\,dt,\qquad
 I_q(b,c)=\frac1{B(b,c)}\int_0^q t^{b-1}(1-t)^{c-1}\,dt.
 \label{eq:beta-functions}
\end{equation}

\begin{proposition}
\label{prop:two-point}
Let $u\ne v$ be cable points and $0<\alpha<1$. Then
\begin{equation}
 p_\alpha(u,v)=I_{\rho(u,v)^2}(1-\alpha,\alpha).
 \label{eq:beta}
\end{equation}
More precisely, for $q=\rho(u,v)^2$ and $z>0$,
\begin{equation}
 \Pp_\alpha^T(u\leftrightarrow v\mid X(u)=z)
 =F_{\alpha,q}(z):=\frac1{\Gamma(1-\alpha)}
       \int_0^{zq/(1-q)}r^{-\alpha}e^{-r}\,dr.
 \label{eq:conditional-connection}
\end{equation}
With $F_{\alpha,q}(0)=0$, this function is nonnegative, increasing
and concave on $[0,\infty)$. If
$\beta=2(1-\alpha)$, there are constants
$0<c_\alpha\le C_\alpha<\infty$, depending only on $\alpha$, such that
\begin{equation}
 c_\alpha\rho(u,v)^\beta\le p_\alpha(u,v)
 \le C_\alpha\rho(u,v)^\beta.
 \label{eq:power}
\end{equation}
The bounds in \eqref{eq:power} also hold for $u=v$. For $\alpha\ge1$,
\begin{equation}
 \Pp_\alpha^T(\wt T\text{ is a single loop cluster})=1.
 \label{eq:full-connection}
\end{equation}
\end{proposition}

\begin{proof}
By Lemma~\ref{lem:branching} and \eqref{eq:positive-path}, the
conditional connection probability is the probability that the
diffusion \eqref{eq:CIR}, started at $z$, avoids zero up to time
$-\log q$. By \eqref{eq:bessel-time-change}, this is the probability
that $Y$ avoids zero up to $(1-q)/q$. For $Y_0=z>0$, let
$T_0=\inf\{s\ge0:Y_s=0\}$. The squared-Bessel hitting-time formula gives
\begin{equation}
 T_0\overset d=\frac{z}{G_{1-\alpha}},
 \qquad G_{1-\alpha}\sim\Ga(1-\alpha,1);
 \label{eq:hit-zero}
\end{equation}
see~\cite[Proposition~2.9]{LawlerBessel}, applied to $\sqrt{2Y}$
with Bessel parameter $\alpha-1/2$. This proves
\eqref{eq:conditional-connection}. For $z>0$, the derivative of
$F_{\alpha,q}$ is a positive constant times
$z^{-\alpha}e^{-zq/(1-q)}$, which is decreasing. Together with
$F_{\alpha,q}(0)=0$, this proves concavity.

Unconditionally, $X(u)\sim\Ga(\alpha,1)$. Let
$G_\alpha\sim\Ga(\alpha,1)$ be independent of $G_{1-\alpha}$ from
\eqref{eq:hit-zero}. Integrating \eqref{eq:conditional-connection} gives
\begin{equation}
 p_\alpha(u,v)
 =\Pp\left(G_{1-\alpha}<\frac{qG_\alpha}{1-q}\right)
 =\Pp\left(\frac{G_{1-\alpha}}{G_\alpha+G_{1-\alpha}}<q\right).
 \label{eq:gamma-ratio}
\end{equation}
The ratio in \eqref{eq:gamma-ratio} has the
$\operatorname{Beta}(1-\alpha,\alpha)$ law, which proves \eqref{eq:beta}.

In the beta integral, $(1-t)^{\alpha-1}\ge1$. For $q\le1/2$ it
is also at most $2^{1-\alpha}$ on $[0,q]$. For $q>1/2$, use
$I_q(1-\alpha,\alpha)\le1\le2^{1-\alpha}q^{1-\alpha}$. These estimates give
\eqref{eq:power}, with constants adjusted to include $u=v$.

If $\alpha\ge1$, the squared-Bessel process started above zero
never hits zero; see~\cite[Proposition~2.1]{LawlerBessel}.
Every neighboring pair of vertices is therefore connected almost
surely. There are countably many edges, and their connecting
geodesics cover $\wt T$. This proves \eqref{eq:full-connection}.
\end{proof}

At $\alpha=1/2$, \eqref{eq:beta} reduces to
\eqref{eq:lupu-two-point}.

\paragraph*{A comparison at the branching point.}

The following elementary inequality controls the dependence produced
by the occupation at a common branching point.

\begin{lemma}
\label{lem:concavity}
Let $\alpha>0$ and let $k\ge1$ be an integer. If
$Z\sim\Ga(\alpha,1)$ and $f_1,\ldots,f_k$ are bounded
nonnegative increasing concave functions on $[0,\infty)$, then
\begin{equation}
 \E\prod_{i=1}^k f_i(Z)
 \le\frac{\E Z^k}{(\E Z)^k}\prod_{i=1}^k\E f_i(Z).
 \label{eq:concavity}
\end{equation}
\end{lemma}

\begin{proof}
For any $f=f_i$, concavity and nonnegativity at zero imply that $f(z)/z$ is
decreasing for $z>0$. Under the size-biased law
$dQ=Z\,d\Pp/\E Z$, it has nonpositive covariance with every
increasing function $\psi$ for which the terms are integrable.
Equivalently,
\begin{equation}
 \E[f(Z)\psi(Z)]\le\frac{\E f(Z)}{\E Z}\E[Z\psi(Z)].
 \label{eq:size-biased-comparison}
\end{equation}
Apply \eqref{eq:size-biased-comparison} successively to
$f_1,\ldots,f_k$. At the $i$th step,
$\psi$ is a product of the remaining increasing nonnegative factors
and $Z^{i-1}$, so it is increasing. Gamma moments ensure
integrability. This yields \eqref{eq:concavity}.
\end{proof}

\begin{proposition}
\label{prop:quasi}
Let $0<\alpha<1$ and write $p(x)=p_\alpha(o,x)$ for $x\in V$.
For $x,y\in V$ and $a=x\wedge y$,
\begin{equation}
 \Pp_\alpha^T(o\leftrightarrow x,\ o\leftrightarrow y)
 \le K_\alpha\frac{p(x)p(y)}{p(a)},\qquad
 K_\alpha=\frac{(\alpha+1)(\alpha+2)}{\alpha^2}.
 \label{eq:quasi}
\end{equation}
\end{proposition}

\begin{proof}
Suppose first that $o,x,y$ lie in three different components of
$\wt T\setminus\{a\}$. For $v\in\{o,x,y\}$, set
$f_v(z)=F_{\alpha,\rho(a,v)^2}(z)$. Conditional on $Z=X(a)$,
the three events $\{a\leftrightarrow v\}$ are independent with
increasing concave probabilities $f_v(Z)$ by
\eqref{eq:conditional-connection}. Their simultaneous
occurrence is precisely the joint event in \eqref{eq:quasi}.
Equation~\eqref{eq:concavity}, with
$\E Z^3/(\E Z)^3=K_\alpha$, gives
\begin{equation}
 \Pp_\alpha^T(o\leftrightarrow x,\ o\leftrightarrow y)
 \le K_\alpha p_\alpha(o,a)p_\alpha(a,x)p_\alpha(a,y).
 \label{eq:three-connections}
\end{equation}
Two increasing functions of the same scalar variable have
nonnegative covariance. Therefore
\begin{equation}
 p(x)=\E[f_o(Z)f_x(Z)]
 \ge p_\alpha(o,a)p_\alpha(a,x),
 \label{eq:pair-positive-correlation}
\end{equation}
and the same inequality holds with $y$ in place of $x$. Substitution
in \eqref{eq:three-connections} proves \eqref{eq:quasi}. If $a=o$,
apply \eqref{eq:concavity} to $f_x$ and $f_y$; the moment constant
$(\alpha+1)/\alpha$ is at most $K_\alpha$.
If $a=x$ or $a=y$, the connection events are nested and the bound
holds with constant one.
\end{proof}

Thus cable connectivity is quasi-independent in the sense
of~\cite[Section~5.4]{LP}. A related three-point estimate for
discrete loop clusters on regular trees is proved
in~\cite[Lemma~5.1]{MS}. The constant in \eqref{eq:quasi} depends
only on the soup intensity.

\subsection{Boundary dimension and finite-energy flows}
\label{sec:cutsets}

For a unit flow $\mu$ and $\beta>0$, define its boundary energy by
\begin{equation}
 \mathcal E_\beta(\mu)=
 \iint_{\partial T\times\partial T}
      \delta_G(\xi,\eta)^{-\beta}\,d\mu(\xi)d\mu(\eta),
 \label{eq:energy}
\end{equation}
with integrand $+\infty$ when $\xi=\eta$. By
\eqref{eq:green-boundary-metric}, the kernel is
$\rho_o(\xi\wedge\eta)^{-\beta}$ off the diagonal, or equivalently
$g(o,o)^{\beta/2}\Theta_o(\xi,\eta)^{\beta/2}$ by
\eqref{eq:green-naim}. Thus this is the usual inverse-distance
energy for the Green boundary. By
\eqref{eq:ray-correlation-decay}, telescoping along the common
ancestors and then applying monotone convergence gives
\begin{equation}
 \mathcal E_\beta(\mu)
 =1+\sum_{x\ne o}
    [\rho_o(x)^{-\beta}-\rho_o(x^-)^{-\beta}]\mu(x)^2.
 \label{eq:energy-series}
\end{equation}

\begin{lemma}
\label{lem:capacity-criterion}
Let $0<\alpha<1$ and $\beta=2(1-\alpha)$. If
\begin{equation}
 \inf_\Pi\sum_{x\in\Pi}\rho_o(x)^\beta=0,
 \label{eq:zero-cutset-sum}
\end{equation}
there is almost surely no unbounded cluster.
If $\mathcal E_\beta(\mu)<\infty$ for some unit
flow, then
\begin{equation}
 \Pp_\alpha^T(o\leftrightarrow\infty)
 \ge \frac{c_\alpha}{K_\alpha\mathcal E_\beta(\mu)}>0.
 \label{eq:survival}
\end{equation}
\end{lemma}

\begin{proof}
An infinite root cluster contains a ray and so meets every cutset.
The union bound and \eqref{eq:power} imply
\begin{equation}
 \Pp_\alpha^T(o\leftrightarrow\infty)
 \le C_\alpha\sum_{x\in\Pi}\rho_o(x)^\beta.
 \label{eq:cutset-first-moment}
\end{equation}
Under \eqref{eq:zero-cutset-sum}, taking the infimum in
\eqref{eq:cutset-first-moment} proves that the root cluster is bounded
almost surely.
Every unbounded cluster contains a vertex. If one occurred with
positive probability, countability would give a fixed vertex $u$
with positive probability of belonging to one. The events
$\{o\leftrightarrow u\}$ and $\{u\leftrightarrow\infty\}$ are
increasing in the Poisson soup, and are positively associated.
The Poisson Harris inequality gives
\begin{equation}
 \Pp_\alpha^T(o\leftrightarrow\infty)
 \ge p_\alpha(o,u)\Pp_\alpha^T(u\leftrightarrow\infty).
 \label{eq:root-transfer}
\end{equation}
Since $p_\alpha(o,u)>0$, \eqref{eq:root-transfer} rules out an
unbounded cluster at any fixed vertex, and hence anywhere in $\wt T$.

To prove \eqref{eq:survival}, we give the weighted second-moment
argument directly; see~\cite[Section~5.3 and Theorem~5.19]{LP}.
Fix a unit flow $\mu$ with $\mathcal E_\beta(\mu)<\infty$ and set
\begin{equation}
 Z_n=\sum_{|x|=n}\frac{\mu(x)}{p_\alpha(o,x)}
             \ind_{\{o\leftrightarrow x\}}.
 \label{eq:weighted-connection-count}
\end{equation}
Flow conservation gives $\E_\alpha^T Z_n=1$. By
\eqref{eq:quasi} and \eqref{eq:power},
\begin{equation}
 \begin{aligned}
 \E_\alpha^T[Z_n^2]
 &\le K_\alpha\sum_{|x|=|y|=n}
       \frac{\mu(x)\mu(y)}{p_\alpha(o,x\wedge y)}\\
 &\le\frac{K_\alpha}{c_\alpha}
       \sum_{|x|=|y|=n}\mu(x)\mu(y)\rho_o(x\wedge y)^{-\beta}
 \le\frac{K_\alpha}{c_\alpha}\mathcal E_\beta(\mu).
 \end{aligned}
 \label{eq:connection-second-moment}
\end{equation}
For the last inequality, if $\xi\in\partial T_x$ and
$\eta\in\partial T_y$, then
$\delta_G(\xi,\eta)\le\rho_o(x\wedge y)$. This also holds when
$x=y$, so the diagonal terms are included in the energy bound.
Let $A_n$ be the event that the root cluster meets generation $n$.
Since $\{Z_n>0\}\subseteq A_n$, Cauchy--Schwarz gives
\begin{equation}
 \Pp_\alpha^T(A_n)\ge\Pp_\alpha^T(Z_n>0)
 \ge\frac{(\E_\alpha^T Z_n)^2}{\E_\alpha^T[Z_n^2]}
 \ge\frac{c_\alpha}{K_\alpha\mathcal E_\beta(\mu)}.
 \label{eq:generation-second-moment}
\end{equation}
By connectedness, local finiteness and properness, the events $A_n$
decrease to $\{o\leftrightarrow\infty\}$. Letting $n\to\infty$
proves \eqref{eq:survival}.
\end{proof}

\begin{lemma}
\label{lem:frostman}
Under \eqref{eq:weighted-setting}, $\delta_G$ is an ultrametric
inducing the compact cylinder topology on $\partial T$, and
\eqref{eq:green-boundary-dimension} holds. If $0<\beta<s<D_G$,
there is a unit flow $\mu$ such that
\begin{equation}
 \mathcal E_\beta(\mu)\le1+\delta^{-1}\frac{\beta}{s-\beta},
 \qquad \delta=\inf_\Pi\sum_{x\in\Pi}\rho_o(x)^s>0.
 \label{eq:frostman}
\end{equation}
\end{lemma}

\begin{proof}
The strict decrease of $\rho_o$ along paths and
\eqref{eq:ray-correlation-decay} show that $\delta_G$ is an
ultrametric. For a ray $\xi$ and $0<t<1$, the closed ball
$B_{\delta_G}(\xi,t)$ is $\partial T_{\xi_n}$, where $n$ is the first index
with $\rho_o(\xi_n)\le t$. Hence its topology is the cylinder
topology, which is compact by local finiteness.
If $s>D_G$, choose cutsets whose sums of $\rho_o(x)^s$ tend to
zero. Their cylinders cover $\partial T$ and have diameter at
most $\rho_o(x)$. The largest diameter also tends to zero, so the
$s$-dimensional Hausdorff measure is zero. If $D_G=0$, this already
proves \eqref{eq:green-boundary-dimension}.

Now fix $0<\beta<s<D_G$ and let $\delta$ be as in
\eqref{eq:frostman}.
Since $0<\rho_o(x)\le1$, the set of exponents with positive cutset
infimum is downward closed, so $\delta>0$. The max-flow/min-cut
theorem~\cite[Theorem~3.1]{LP}, with vertex capacities
$\rho_o(x)^s$, gives a unit flow satisfying
\begin{equation}
 \mu(x)\le\delta^{-1}\rho_o(x)^s\qquad(x\in V).
 \label{eq:flow-correlation-bound}
\end{equation}
The description of balls above gives
$\mu(B_{\delta_G}(\xi,t))\le\delta^{-1}t^s$ for $0<t<1$.
Every set of positive diameter is contained in a closed ball of
that diameter, and singletons have zero mass by
\eqref{eq:ray-correlation-decay}. Thus every countable cover of
the boundary by sets of diameter less than one has sum of
$s$th powers of diameters at least $\delta$. This proves
$\dim_{\mathrm H}(\partial T,\delta_G)\ge s$, and establishes
\eqref{eq:green-boundary-dimension} by letting $s\uparrow D_G$.

For $\xi\in\partial T$, put $t_n=\rho_o(\xi_n)^{-\beta}$, so
$t_0=1$ and $t_n\uparrow\infty$ by \eqref{eq:ray-correlation-decay}.
Using \eqref{eq:flow-correlation-bound} in \eqref{eq:energy-series}
and applying Tonelli's theorem gives
\begin{equation}
\begin{aligned}
 \mathcal E_\beta(\mu)-1
 &\le\delta^{-1}\sum_{x\ne o}\mu(x)\rho_o(x)^s
        [\rho_o(x)^{-\beta}-\rho_o(x^-)^{-\beta}]\\
 &=\delta^{-1}\int_{\partial T}
      \sum_{n\ge1}t_n^{-s/\beta}(t_n-t_{n-1})\,d\mu(\xi),
\end{aligned}
 \label{eq:energy-along-rays}
\end{equation}
where the sequence $(t_n)$ in the integral depends on $\xi$.
Since $t\mapsto t^{-s/\beta}$
is decreasing,
\begin{equation}
 \sum_{n\ge1}t_n^{-s/\beta}(t_n-t_{n-1})
 \le\int_1^\infty t^{-s/\beta}\,dt=\frac{\beta}{s-\beta}.
 \label{eq:energy-integral-bound}
\end{equation}
Substituting \eqref{eq:energy-integral-bound} into
\eqref{eq:energy-along-rays} proves \eqref{eq:frostman}.
\end{proof}

\begin{proof}[Proof of Theorem~\ref{thm:classification}]
Definition~\eqref{eq:DG} gives $D_G\ge0$. To prove the upper bound, combine
\eqref{eq:green-resistance} and \eqref{eq:energy-series}:
\begin{equation}
 \mathcal E_2(\mu)
 =1+\frac1{g_o}\sum_{x\ne o}
              \frac{\mu(x)^2}{\widehat c_{x^-x}}.
 \label{eq:energy-two}
\end{equation}
If $D_G>2$, \eqref{eq:frostman} with $\beta=2$ would give
a finite-energy unit flow in the recurrent network
$(T,\widehat c)$, a contradiction. Thus $D_G\le2$.

For $0<\alpha<1$, put $\beta=2(1-\alpha)$.
If $\beta<D_G$, choose $\beta<s<D_G$ and apply
\eqref{eq:frostman} and \eqref{eq:survival} to get
positive probability of root percolation. If $\beta>D_G$, the
cutset infimum at exponent $\beta$ is zero, so
\eqref{eq:cutset-first-moment} and \eqref{eq:root-transfer} give
non-percolation.
Equation~\eqref{eq:full-connection} handles $D_G=0$ and all
intensities at least one. If $D_G=2$, the implication from
$\beta<D_G$ applies at every
$0<\alpha<1$. These statements prove \eqref{eq:classification}
throughout $0\le D_G\le2$.

The proof applies with any chosen root. Since the critical intensity
in \eqref{eq:critical} is defined without a root,
\eqref{eq:classification} proves that $D_G(T,o)$ is independent of $o$.
Equation~\eqref{eq:root-transfer}, applied in both directions, also
makes positive percolation probability equivalent at all fixed
vertices. Independence of the skeleton follows from
\eqref{eq:point-cutsets}.
\end{proof}

\Needspace{12\baselineskip}
\subsection{Metric formulation and weighted examples}
\label{sec:metric}

\paragraph*{Point cutsets and the real-tree formulation.}

\begin{lemma}
\label{lem:point-cutsets}
On a cable tree satisfying \eqref{eq:weighted-setting}, let
$\mathcal C_o^{\mathrm{pt}}$ be the finite point cutsets of its
infinite rays. For every $s\ge0$,
\begin{equation}
 \inf_{\Pi\in\mathcal C_o^{\mathrm{pt}}}\sum_{x\in\Pi}\rho_o(x)^s
 =\inf_{\Pi\in\mathcal C_o}\sum_{x\in\Pi}\rho_o(x)^s.
 \label{eq:point-cutsets}
\end{equation}
Consequently $D_G$ is independent of the weighted skeleton.
\end{lemma}

\begin{proof}
Every vertex cutset is a point cutset, so the point infimum is no
larger. Conversely, start with a finite point cutset and discard
points belonging to no infinite ray. Move each point in the interior
of an edge to its endpoint farther from $o$. All rays previously
met are still met. By \eqref{eq:multiplicative}, each weight
$\rho_o(x)^s$ can only decrease. If necessary, discard duplicates
and descendants of other selected vertices so that no two remaining
vertices are ancestors of one another;
the remaining set still meets every ray. Thus the vertex infimum is
no larger than the point infimum. This proves
\eqref{eq:point-cutsets}. The left-hand side of
\eqref{eq:point-cutsets} uses only the metric tree and its Brownian
Green function. Both the length measure and the energy
\eqref{eq:brownian-energy} are unchanged when degree-two vertices
are inserted or removed. The exponent is therefore the same for
every weighted skeleton of this metric tree.
\end{proof}

\begin{proof}[Proof of Corollary~\ref{cor:real-tree}]
By the skeleton theorem of
Athreya--Eckhoff--Winter~\cite[Lemma~7.1 and Remark~7.2]{AEW},
every point of $\mathcal T$ has finite degree, meaning that its removal
leaves finitely many connected components. Each bounded ball contains
only finitely many branch points, that is, points of degree at least
three. Include $o$ among these vertices
and insert degree-two vertices along any unbounded arcs. We obtain
a locally finite skeleton with finite positive edge lengths.

Assign conductance $c_e=(2\ell_e)^{-1}$ to an edge of length
$\ell_e$. The cable energy $\frac12\int|f'|^2$ is exactly the
Brownian energy on $\mathcal T$. For fixed values at the endpoints
$x,y$ of $e$, the least energy on that edge is
$c_e(f(x)-f(y))^2$, attained by linear interpolation.
Thus the cable diffusion, its Green
function and its loop soup agree with those on $\mathcal T$.
Equations~\eqref{eq:classification}, \eqref{eq:full-connection} and
\eqref{eq:point-cutsets} give \eqref{eq:real-tree-classification}
and full connection at intensity one.
\end{proof}

The assumption that every point lies on an infinite ray is used
only in this skeleton reduction.

\paragraph*{Spherical cutsets and thresholds below one-half.}

For a spherically symmetric tree with no leaves, let $N_n$ be the size of generation
$n$ and let $\rho_n$ be the common value of $\rho_o(x)$ for
$|x|=n$. The probability measure that chooses a child uniformly at
each step assigns mass $N_n^{-1}$ to each cylinder at depth $n$.
Thus for every $s>0$,
\begin{equation}
 \inf_\Pi\sum_{x\in\Pi}\rho_o(x)^s
 =\inf_{n\ge0}N_n\rho_n^s,\qquad
 D_G=\liminf_{n\to\infty}\frac{\log N_n}{-\log\rho_n}.
 \label{eq:spherical-cutsets}
\end{equation}
Indeed, write each cutset sum as
$\sum_{x\in\Pi}N_{|x|}^{-1}[N_{|x|}\rho_{|x|}^s]$ and apply
\eqref{eq:flow-conservation}. Taking a generation as cutset gives
the reverse inequality. By \eqref{eq:ray-correlation-decay},
$-\log\rho_n\to\infty$. Hence exponents smaller than the liminf in
\eqref{eq:spherical-cutsets} have positive cutset infimum, whereas
exponents larger than that liminf have infimum zero. This proves
the formula for $D_G$ in \eqref{eq:spherical-cutsets}.

\begin{proof}[Proof of Proposition~\ref{prop:weighted-example}]
Every cable has length at least $1/2$, so the metric is proper.
The electrical unit flow splits equally between descendants, and
its potential at generation $n$ is
\begin{equation}
 V_n=\frac1{b+1}\sum_{k=n}^\infty(ba)^{-k}
     =g_o(ba)^{-n},\qquad g_o=\frac{ba}{(b+1)(ba-1)}<\infty.
 \label{eq:geometric-potential}
\end{equation}
The convergent series of level resistances in \eqref{eq:geometric-potential}
also proves transience.
Consequently $h(x)=(ba)^{-n}$ for $|x|=n$. Equation
\eqref{eq:green-resistance} gives the exact identity
\begin{equation}
 \rho_n^{-2}
 =1+(b+1)(ba-1)\frac{(b^2a)^n-1}{b^2a-1}.
 \label{eq:geometric-correlation}
\end{equation}
Hence $\rho_n\asymp(b\sqrt a)^{-n}$. Since
$N_n=(b+1)b^{n-1}$ for $n\ge1$, equation
\eqref{eq:spherical-cutsets} proves the formula for $D_G$ in
\eqref{eq:weighted-example}; \eqref{eq:classification} gives the
critical intensity.
\end{proof}

\begin{proof}[Proof of Proposition~\ref{prop:zero-threshold}]
We work with the unsplit binary tree, giving conductance
$c_n=L_n^{-1}$ to each edge between generations $n$ and $n+1$.
Its cable tree is proper since every cable has length at least
$1/2$, and subdivision leaves the cable tree unchanged.
The equally splitting unit flow has energy $V_0<\infty$ and
potential at generation $n$ given by
\begin{equation}
 V_n=\sum_{k=n}^\infty\frac{L_k}{2^{k+1}}
     \asymp\frac1{n+1}.
 \label{eq:zero-threshold-potential}
\end{equation}
Indeed, the ceiling in \eqref{eq:zero-threshold-lengths} changes this
tail by at most $2^{-n}$. In particular, the network is transient.
Writing $h_n=V_n/V_0$, the transformed conductances satisfy
\begin{equation}
 \widehat c_n=L_n^{-1}h_nh_{n+1}\asymp2^{-n}.
 \label{eq:zero-threshold-transform}
\end{equation}
Equation~\eqref{eq:green-resistance} therefore gives
\begin{equation}
 \rho_n^{-2}=1+\frac1{V_0}\sum_{k=0}^{n-1}\widehat c_k^{-1}
       \asymp2^n.
 \label{eq:zero-threshold-correlation}
\end{equation}
Since $N_n=2^n$, the spherical formula \eqref{eq:spherical-cutsets}
yields $D_G=2$. The classification and skeleton invariance in
Theorem~\ref{thm:classification} prove \eqref{eq:zero-threshold}.
\end{proof}

\subsection{Capacity, harmonic measure and deterministic criteria}
\label{sec:unit}

Throughout this subsection we assume \eqref{eq:weighted-setting}.
Additional unit-conductance or degree assumptions are stated where
they are used. We first give the reductions and analytic criteria
needed for Theorem~\ref{thm:half-analytic}.

\paragraph*{Capacity along a ray.}

For adjacent vertices $u,v$, let $C_{u,v}$ be the conductance from
$u$ to infinity through the edge $uv$ and the component beyond $v$.
For a child $y$ of $x$, the series law gives
\begin{equation}
 C_{x,y}=\frac{c_{xy}C_y^+}{c_{xy}+C_y^+},\qquad
 C_x^+=\sum_{y:y^-=x}C_{x,y}.
 \label{eq:forward-conductances}
\end{equation}
For a ray $\xi$, write
\begin{equation}
 s_n(\xi)=\sum_{\substack{y:y^-=\xi_n\\y\ne\xi_{n+1}}}
                        C_{\xi_n,y}
 \label{eq:side-conductance}
\end{equation}
for the total conductance of its off-ray branches at $\xi_n$.

\begin{proposition}
\label{prop:ray-capacity}
For every $\xi\in\partial T$ and $n\ge1$,
\begin{equation}
 \capacity([o,\xi_n])=\sum_{j=0}^{n-1}s_j(\xi)+C_{\xi_n}^+.
 \label{eq:segment-capacity}
\end{equation}
Consequently \eqref{eq:capacity-condition} is equivalent to divergence
of the right-hand side of \eqref{eq:segment-capacity} along every
infinite ray. If all conductances are one, this is equivalent to
\begin{equation}
 \sum_{n\ge0}s_n(\xi)=\infty\qquad(\xi\in\partial T).
 \label{eq:ray-capacity-condition}
\end{equation}
\end{proposition}

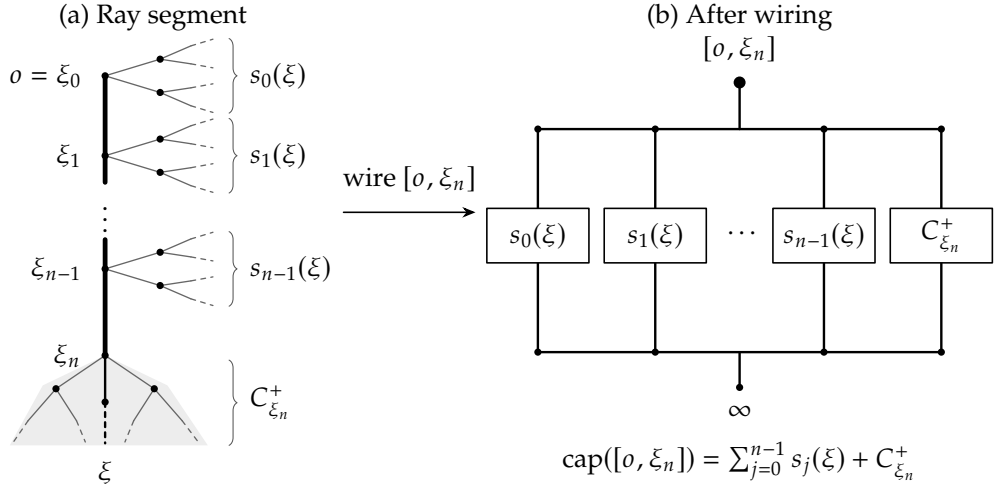
\begin{figure}[htbp]
 \centering
 \input{figures/cable_ray_capacity.tex}
 \caption{Wiring a ray segment. The off-ray branches at $\xi_j$, $j<n$,
 have total conductance $s_j(\xi)$. The shaded forward subtree at
 $\xi_n$ includes the continuation of the ray and contributes
 $C_{\xi_n}^+$. Wiring the bold segment places these conductances
 in parallel to infinity.}
 \label{fig:ray-capacity}
\end{figure}

\begin{proof}
Wire the finite path $[o,\xi_n]$ to one vertex, as in
Figure~\ref{fig:ray-capacity}. Its exterior
branches are disjoint and in parallel. The branches leaving before
$\xi_n$ contribute the sum in \eqref{eq:segment-capacity}, and
those leaving the endpoint contribute $C_{\xi_n}^+$. Requiring
potential one at the path vertices or on its whole cable gives the
same capacity, since the minimizing function is constant along
the intervening edges. This proves \eqref{eq:segment-capacity}.
Exhaustion gives the capacity of the full ray. Every unbounded
closed connected subset contains a ray tail, and a finite initial
segment has finite capacity, so this tests
\eqref{eq:capacity-condition}. In the unit case,
\begin{equation}
 C_{\xi_n}^+=s_n(\xi)+C_{\xi_n,\xi_{n+1}},\qquad
                         0\le C_{\xi_n,\xi_{n+1}}\le1.
 \label{eq:unit-ray-remainder}
\end{equation}
Combining \eqref{eq:segment-capacity} and
\eqref{eq:unit-ray-remainder} proves the equivalence with
\eqref{eq:ray-capacity-condition}.
\end{proof}

\paragraph*{Finite branches.}

Let
\begin{equation}
 T^\infty=\{x\in V:|T_x|=\infty\}
 \label{eq:infinite-descendants}
\end{equation}
be the subtree of vertices with infinite descent. Each component
removed from $T$ is finite and has one attachment to $T^\infty$.

\begin{lemma}
\label{lem:pruning}
The tree $T^\infty$, with its inherited conductances, is transient,
its cable metric is proper, and
\begin{equation}
 g_T(x,y)=g_{T^\infty}(x,y)\quad(x,y\in T^\infty),\qquad
 D_G(T)=D_G(T^\infty).
 \label{eq:pruning}
\end{equation}
Capacities of finite subsets of $T^\infty$ agree in both networks.
Condition~\eqref{eq:capacity-condition} holds on $T$ if and only if
it holds on $T^\infty$.
\end{lemma}

\begin{proof}
Extend a finitely supported function on $T^\infty$ constantly over
every attached finite branch. Only finitely many attachment vertices
have nonzero values, and each has finitely many finite branches, so
the extension still has finite support. It adds zero energy.
Conversely, restriction from $T$ to $T^\infty$ cannot increase
energy. Thus minimizing the energy over extensions of a function
on $T^\infty$ gives exactly its energy in the pruned tree.
The variational characterizations
give transience, equality of Green functions and equality of
capacities. The same argument proves transience of $T$ whenever
$T^\infty$ is transient. This is the finite-branch reduction used
also in~\cite[Section~1, (1.31)--(1.40)]{AS}. The inherited cable
tree is closed in $\wt T$, so it is proper.

The infinite rays are unchanged. Removing vertices outside
$T^\infty$ from a cutset leaves a cutset and cannot increase its
weight; conversely every cutset of $T^\infty$ is one of $T$.
Equality of Green functions proves equality of the infima in
\eqref{eq:DG}. The rays and their finite-segment capacities agree
in both trees, so Proposition~\ref{prop:ray-capacity} proves the
equivalence of \eqref{eq:capacity-condition}.
\end{proof}

\paragraph*{Currents and the analytic criterion.}

For $x\ne o$, let $a_x=C_{x,x^-}$ be the conductance from $x$ to
infinity through its parent side. The parallel law and Ohm's law give
\begin{equation}
 g(x,x)=\frac1{C_x^++a_x},\qquad
 \theta(x)=C_x^+g(o,x),\qquad
 \Lambda(x)=\frac{(C_x^+)^2}{C_x^++a_x}\le C_x^+
                       \quad(\theta(x)>0).
 \label{eq:general-current}
\end{equation}
These prove \eqref{eq:electrical-flow} and
\eqref{eq:current-factor}; when $C_x^+=0$, the current is zero.
The flow has finite electrical energy $g(o,o)$. Every infinite
ray has infinite total resistance, since its cable length is
infinite by properness. If its current had a positive limiting
value, the energy along that ray would be infinite. Hence
\begin{equation}
 \theta(\xi_n)\longrightarrow0\qquad(\xi\in\partial T).
 \label{eq:harmonic-nonatomic}
\end{equation}
In particular, the denominator of
\eqref{eq:analytic-ray-condition} tends to infinity on every ray
whose vertices have positive current.

\begin{proof}[Proof of Theorem~\ref{thm:half-analytic}]
Theorem~1.1(1) of~\cite{DPR} and
\eqref{eq:capacity-condition} imply that cable Gaussian sign
clusters are bounded. By Lupu's coupling~\cite[Proposition~2.1]{Lupu16},
the same holds for loop clusters at intensity $1/2$. Monotonicity
gives $\ac\ge1/2$, and \eqref{eq:classification} gives $D_G\le1$.
Now \eqref{eq:half-condition} follows from the definition of $D_G$
and \eqref{eq:classification}.

By \eqref{eq:energy-series}, the electrical energy of a unit flow
in the network \eqref{eq:analytic-resistances} is
$\mathcal E_\beta(\mu)-1$. If the cutset condition in
\eqref{eq:half-condition} holds, then $D_G\ge1$, and
Lemma~\ref{lem:frostman} gives a finite-energy unit flow for every
$\beta<1$. Conversely, for any unit flow with finite
$\mathcal E_\beta(\mu)$ and any cutset $\Pi$, the pairs of rays
passing through the same $x\in\Pi$ give, by Cauchy--Schwarz,
\begin{equation}
 \mathcal E_\beta(\mu)\ge
 \sum_{x\in\Pi}\frac{\mu(x)^2}{\rho_o(x)^\beta}
 \ge\left(\sum_{x\in\Pi}\rho_o(x)^\beta\right)^{-1}.
 \label{eq:energy-cutset-converse}
\end{equation}
Electrical transience is equivalent to existence of a finite-energy
unit flow~\cite[Theorem~2.11]{LP}. This proves the claimed equivalence.

Suppose now that \eqref{eq:analytic-ray-condition} holds on a set
of positive harmonic measure. Fix $0<s<1$ and choose $\eta>0$
with $s(1+\eta/2)\le1$. By \eqref{eq:harmonic-nonatomic}, for
some finite $A\ge1$, the set
\begin{equation}
 B_A=\{\xi:\theta(\xi_n)>0,\quad
       \Lambda(\xi_n)\le A\theta(\xi_n)^{-\eta}
                         \text{ for every }n\ge1\}
 \label{eq:analytic-good-rays}
\end{equation}
has positive harmonic measure. Restrict $\theta$ to $B_A$ and
normalize to obtain a unit flow $\mu$. At every vertex with
$\mu(x)>0$, \eqref{eq:current-factor} implies
\begin{equation}
 \rho_o(x)^s\ge (Ag(o,o))^{-s/2}\theta(x)^{s(1+\eta/2)}
 \ge (Ag(o,o))^{-s/2}\theta(B_A)\mu(x).
 \label{eq:analytic-flow-bound}
\end{equation}
The root cutset has weight one. All other cutsets have weight at
least $(Ag(o,o))^{-s/2}\theta(B_A)>0$ by conservation of $\mu$.
This proves \eqref{eq:half-condition} for every $s<1$.
\end{proof}

\begin{proposition}
\label{prop:degree}
Assume \eqref{eq:capacity-condition}. If there are $\kappa\ge0$
and $0<A<\infty$ such that
\begin{equation}
 \Lambda(x)\le A\theta(x)^{-\kappa}
             \qquad(x\ne o,\ \theta(x)>0),
 \label{eq:conductance-assumption}
\end{equation}
then
\begin{equation}
 D_G(T)\ge\frac2{2+\kappa},\qquad
 \frac12\le\ac(T)\le\frac{1+\kappa}{2+\kappa}.
 \label{eq:degree-conclusion}
\end{equation}
For unit conductances it suffices that
\begin{equation}
 \de(x)\le A\theta(x)^{-\kappa}
             \qquad(x\ne o,\ \theta(x)>0).
 \label{eq:degree-assumption}
\end{equation}
If \eqref{eq:conductance-assumption} holds for every $\kappa>0$,
with $A$ depending on $\kappa$, then $\ac(T)=1/2$.
\end{proposition}

\begin{proof}
Equation~\eqref{eq:current-factor} gives
$\rho_o(x)\ge c\theta(x)^{1+\kappa/2}$ on the support of
$\theta$, where $c=(Ag(o,o))^{-1/2}$. At
$s=2/(2+\kappa)$, summing over any cutset distinct from $\{o\}$
gives weight at least $c^s\sum_{x\in\Pi}\theta(x)=c^s$.
The root cutset has weight one. This proves the lower bound for
$D_G$; Theorem~\ref{thm:half-analytic} and
\eqref{eq:classification} prove the intensity bounds. For unit
conductances, \eqref{eq:general-current} gives
$\Lambda(x)\le C_x^+\le\de(x)-1$, proving the degree assertion.
\end{proof}

For unit conductances, Theorem~\ref{thm:half-analytic} also gives
$D_G(T)=1$ and $\ac(T)=1/2$ under
\eqref{eq:capacity-condition} whenever a set of rays of positive
harmonic measure satisfies
\begin{equation}
 \frac{\log\de(\xi_n)}{\log(1/\theta(\xi_n))}
                        \longrightarrow0.
 \label{eq:ray-assumption}
\end{equation}
\paragraph*{Unit conductances and minimum degree three.}

Under \eqref{eq:standing}, every forward branch contains a binary
unit tree, which implies transience; fixed cable lengths give properness.

\begin{lemma}
\label{lem:electrical}
Under \eqref{eq:standing}, write $d_x=\de(x)$. For $x\ne o$,
and for adjacent vertices $u,v$, we have
\begin{gather}
 \frac12\le C_{u,v}\le1\qquad(u\sim v),
                                                    \label{eq:branch-conductance}\\
 \frac{d_x-1}{2}\le C_x^+\le d_x-1,\qquad
 \frac12\le a_x\le1,                               \label{eq:C}\\
 \frac{\theta(x)^2}{g(o,o)d_x}\le\rho_o(x)^2
       \le\frac{6\theta(x)^2}{g(o,o)d_x}\quad(x\ne o),\label{eq:current}\\
 \rho(u,v)\le\frac12\quad(u\sim v),\qquad
 \rho_o(x)\le2^{-|x|},\qquad
 \theta(x)\le(2/3)^{|x|}.                            \label{eq:decays}
\end{gather}
\end{lemma}

\begin{proof}
Each component beyond an oriented edge contains a rooted binary
unit tree, whose conductance is one. In series with the initial
unit edge it has conductance in $[1/2,1]$. This proves
\eqref{eq:branch-conductance}, and summing the appropriate branches
gives \eqref{eq:C}. The ratio $(C_x^++a_x)/(C_x^+)^2$ lies between
$1/d_x$ and $2/(d_x-1)+4/(d_x-1)^2\le6/d_x$.
Equations~\eqref{eq:current-factor} and \eqref{eq:general-current}
therefore give \eqref{eq:current}.
For adjacent $x,y$, the conductance from $x$ through its other
branches is at least one, so $H(x,y)\le1/2$. Equations
\eqref{eq:hitting} and \eqref{eq:multiplicative} give the bounds
on $\rho$. Finally, at least two forward branches share the incoming
current, each with conductance in $[1/2,1]$. No child receives more
than $2/3$ of that current, proving the bound on $\theta$.
\end{proof}

Every vertex of a ray has an off-ray branch with conductance at
least $1/2$ by \eqref{eq:branch-conductance}. Thus
Proposition~\ref{prop:ray-capacity} verifies
\eqref{eq:capacity-condition}, and Theorem~\ref{thm:half-analytic} gives
\begin{equation}
 \begin{gathered}
 D_G(T)\le1,\qquad \ac(T)\ge\frac12,\\
 \Pp_{1/2}^T(\text{there is an unbounded cluster})=0.
 \end{gathered}
 \label{eq:unit-lower-bound}
\end{equation}
Also under \eqref{eq:standing}, one obtains
\begin{equation}
 \log\max_{|x|=n}\de(x)=o(n)
             \quad\Longrightarrow\quad\ac(T)=\frac12.
 \label{eq:subexponential-degrees}
\end{equation}
Indeed, \eqref{eq:decays} gives
$\log(1/\theta(\xi_n))\ge n\log(3/2)$, so
\eqref{eq:ray-assumption} holds on every ray.

\paragraph*{Spherically symmetric trees with unit conductances.}

\begin{proof}[Proof of Corollary~\ref{cor:symmetric}]
Use the notation in \eqref{eq:symmetric-parameters}.
Symmetry gives $\theta(x)=N_n^{-1}$ at depth $n$.
Equation \eqref{eq:current} yields
$\rho_o(x)=\rho_n\asymp(N_n\sqrt{b_n})^{-1}$, uniformly in $n\ge1$.
Equation~\eqref{eq:spherical-cutsets} now gives
\begin{equation}
 D_G=\liminf_n\frac{\log N_n}{-\log\rho_n}
 =\liminf_n\frac{\log N_n}
                  {\log N_n+\frac12\log b_n+O(1)}
 =\frac2{2+r}.
 \label{eq:symmetric-exponent}
\end{equation}
Equation~\eqref{eq:classification} now gives \eqref{eq:symmetric}.
\end{proof}

\begin{proof}[Proof of Corollary~\ref{cor:fixed-branching}]
Fix $r\in[0,\infty]$. Set $r_k=r$ when $r<\infty$, and $r_k=k$
when $r=\infty$. We take $b_0=3$ and choose increasing exceptional
generations $n_k\ge1$, with offspring counts
\begin{equation}
 b_n=
 \begin{cases}
  \max\{2,\lceil N_n^{r_k}\rceil\},&n=n_k,\\
  2,&n\ge1\text{ and }n\notin\{n_k:k\ge1\}.
 \end{cases}
 \label{eq:sparse-branching}
\end{equation}
The indices are chosen recursively, starting after generation zero.
Before inserting the $k$th exceptional generation, continue with
binary branching until a new index $n_k$ satisfies
\begin{equation}
 0\le\frac{\log N_{n_k}}{n_k}-\log2\le\frac1k.
 \label{eq:binary-recovery}
\end{equation}
Such an index exists: if binary branching continues from a fixed
generation $m$, then $N_n=N_m2^{n-m}$ and
$n^{-1}\log N_n\to\log2$. The lower bound in
\eqref{eq:binary-recovery} holds at every generation since
$N_n\ge3\cdot2^{n-1}$. All offspring counts are finite and at least
two, so the tree satisfies \eqref{eq:standing}.

The uniform cylinder flow used in \eqref{eq:spherical-cutsets},
applied to the weights $\lambda^{-n}$, gives
\begin{equation}
 \operatorname{br}(T)=\liminf_{n\to\infty}N_n^{1/n}=2,
 \qquad
 \limsup_{n\to\infty}\frac{\log b_n}{\log N_n}=r.
 \label{eq:sparse-branching-exponents}
\end{equation}
Here the branching number is two by
\eqref{eq:binary-recovery}. For the second assertion, the quotient tends to
zero along ordinary generations; at $n_k$ it equals $r_k+o(1)$.
Corollary~\ref{cor:symmetric} now gives
$\ac=(1+r)/(2+r)$, with value one when $r=\infty$.
These values exhaust $[1/2,1]$.
\end{proof}

For finite $r$, \eqref{eq:sparse-branching} also attains the upper
bound in \eqref{eq:degree-conclusion} with $\kappa=r$, since
$\theta(x)=N_{|x|}^{-1}$ and $\de(x)\le A N_{|x|}^r$ for a
finite constant $A$. At $r=\infty$, all clusters are bounded for
every fixed $\alpha<1$, whereas the full cable tree is connected
at $\alpha=1$. Thus neither finiteness nor the value of the
branching number prevents the obstruction to percolation above
one-half.

\begin{proof}[Proof of Corollary~\ref{cor:gff-separation}]
By \eqref{eq:capacity-condition} and the absence of killing,
Theorem~1.1(2) of~\cite{DPR} gives $\wt h_*(T)=0$.
\end{proof}

\subsection{Galton--Watson trees}
\label{sec:GW-proof}

We prove Theorem~\ref{thm:GW} using the deterministic criteria from
Subsection~\ref{sec:unit}. We first prune finite branches and verify
\eqref{eq:capacity-condition} for the remaining tree. If its
offspring law has finite logarithmic moment, we control degrees
and currents along harmonic rays to verify
\eqref{eq:analytic-ray-condition}. Otherwise, flows on subtrees
with bounded offspring give the cutset condition in
\eqref{eq:half-condition}.

\paragraph*{The tree of infinite descendants.}

Let $f(s)=\Et s^Z$ and let $q<1$ be the extinction probability of
the ordinary Galton--Watson tree with law $\nu$. Conditional on
survival, its infinite-descendant subtree has offspring law
$\bar\nu$, with generating function
\begin{equation}
 \bar f(s)=\frac{f(q+(1-q)s)-q}{1-q}.
 \label{eq:surviving-law}
\end{equation}
Indeed, given $Z$, the number of children with infinite descent has
law $\operatorname{Binomial}(Z,1-q)$, and one conditions it to be
positive. Conditional on which children survive, their descendant
trees are independent ordinary trees conditioned on survival.
Iterating proves \eqref{eq:surviving-law} and the branching property
without a moment assumption; compare~\cite[(5.3)]{AS}.
In particular $\bar\nu(0)=0$ and
$\bar\nu(1)=f'(q)<1$.

For our root convention, the number $J$ of surviving root children
has the law of $\operatorname{Binomial}(K,1-q)$ conditioned to be
positive, where the binomial mixture is taken before conditioning.
The surviving descendant trees are independent of $J$ and have
offspring law $\bar\nu$. Write $\Pt_\infty$ and $\Et_\infty$ for
this pruned tree law and its expectation, and $Y$ for an offspring
variable with law $\bar\nu$. Thus
\begin{equation}
 1\le Y<\infty\quad\Pt_\infty\text{-almost surely},\qquad
 \Pt_\infty(Y\ge2)>0.
 \label{eq:leafless-law}
\end{equation}
The root has an almost surely finite positive number $J$ of children.

Every descendant tree is transient almost surely. Retaining the
first $\min\{Y,2\}$ children at each vertex gives a Galton--Watson
tree without leaves and with bounded offspring mean
$1+\Pt_\infty(Y\ge2)>1$. Its branching number equals this mean
by~\cite[Corollary~5.10]{LP}. A unit flow bounded exponentially
in depth exists by max-flow/min-cut: for some $\lambda>1$ and
$A<\infty$, it satisfies $\mu(x)\le A\lambda^{-|x|}$.
Its energy is at most $A\sum_{n\ge1}\lambda^{-n}<\infty$,
since the flow through each generation is one. Extend it by zero
to the original descendant tree to prove transience there as well.
By countability, this holds simultaneously at every existing vertex.
The same extension proves transience before pruning. Unit cable
lengths give properness, so Lemma~\ref{lem:pruning} applies.

Until the final proof of Theorem~\ref{thm:GW}, write $T$ for the
tree with law $\Pt_\infty$. Green functions and currents refer to
this whole pruned tree, including when smaller subtrees are retained.
For a nonroot vertex $x$, let $Y_x$ be its number of children and put
\begin{equation}
 W_x=\frac{C_x^+}{1+C_x^+}\in(0,1),\qquad
 C_x^+=\sum_{y:y^-=x}W_y\le Y_x.
 \label{eq:GW-conductance}
\end{equation}
Let $(C,W)$ have the law of the descendant conductance and its
transform $W=C/(1+C)$ in an ordinary tree with offspring law
$\bar\nu$.

\paragraph*{The capacity assumption.}

\begin{lemma}
\label{lem:GW-capacity}
For $\Pt_\infty$-almost every tree,
\eqref{eq:ray-capacity-condition}, and hence
\eqref{eq:capacity-condition}, holds.
\end{lemma}

\begin{proof}
Since $W>0$ almost surely, choose $b>0$ so small that, with
$p=\Pt_\infty(W<b)$,
\begin{equation}
 a=\bar f'(p)=\Et_\infty[Yp^{Y-1}]<1.
 \label{eq:bad-branch-mean}
\end{equation}
This is possible because $p\downarrow0$ as $b\downarrow0$ and
$\bar f'(p)\to\bar\nu(1)<1$. The derivative is finite for
$p<1$, even when the mean is infinite. At $p=0$ the $Y=1$ term
has value one.

Call a downward step $x\to y$ bad if every other child $z$ of
$x$ has $W_z<b$. Let $N_n$ be the number of length-$n$ paths of
bad steps from the root of an ordinary tree with offspring $Y$.
Then
\begin{equation}
 \Et_\infty N_n=a^n.
 \label{eq:bad-path-count}
\end{equation}
Conditional on $Y=k$, sum over the first child. The other $k-1$
descendant trees satisfy $W<b$ independently with probability $p$,
and are independent of the path count in the chosen child's tree.
Thus $\Et_\infty N_n=\Et_\infty[Yp^{Y-1}]\Et_\infty N_{n-1}$,
with $N_0=1$. Tonelli's theorem applies without a finite mean.

An infinite path of bad steps would imply $N_n\ge1$ for every
$n$, an event of probability zero by \eqref{eq:bad-path-count}.
Apply this at every nonroot vertex. Every infinite ray
therefore has infinitely many vertices with an off-ray branch
of conductance at least $b$. This proves
\eqref{eq:ray-capacity-condition}; apply
Proposition~\ref{prop:ray-capacity}.
\end{proof}

\subsubsection{Finite logarithmic moment}
\label{sec:finite-log}

We first obtain a conditional current comparison and an exponential
bound on the masses of cylinders under harmonic measure. Both
estimates hold without a moment assumption.
Write $\theta_T$ for the electrical unit flow in
\eqref{eq:electrical-flow} to emphasize its dependence on the tree.
By \eqref{eq:forward-conductances}, the current at a vertex splits
between its children in proportion to their branch conductances:
\begin{equation}
 \frac{\theta_T(y)}{\theta_T(x)}=\frac{W_y}{C_x^+}
       \quad(y^-=x).
 \label{eq:GW-flow}
\end{equation}

Let $\mathcal U=\bigcup_{n\ge0}\{1,2,\ldots\}^n$ be the set of
Ulam--Harris labels, with $o$ the empty label. Assign independent
offspring $Y_x$ with law $\bar\nu$ to all nonempty labels,
independently of $J$. The root has children $1,\ldots,J$;
an existing nonroot label $x$ with $Y_x=k$ has children $x1,\ldots,xk$.
Juxtaposition denotes concatenation of labels.
For $x\ne o$, put $\mathcal U_x=\{xy:y\in\mathcal U\}$ and define
\begin{equation}
 \mathcal F_x=\sigma(J,Y_y:y\in\mathcal U\setminus\mathcal U_x,
                                      \ y\ne o).
 \label{eq:outside-offspring}
\end{equation}
Existence of $x$ is $\mathcal F_x$-measurable. Put $\theta_T(x)=0$ if
$x$ does not exist; its hypothetical descendant tree and the variable
$W_x$ are still defined by the independent labels. The descendant
variables at and below $x$ are independent of $\mathcal F_x$ and
generate an ordinary Galton--Watson tree with offspring distribution
$\bar\nu$.

\begin{lemma}
\label{lem:terminal}
There is an $\mathcal F_x$-measurable $U_x\in[0,1]$ such that
\begin{equation}
 U_xW_x\le\theta_T(x)\le U_x.
 \label{eq:conditional-flow-comparison}
\end{equation}
In particular, if $\kappa=\Et_\infty W>0$, then, for every
$A\subseteq\{1,2,\ldots\}$ and $n\ge1$,
\begin{equation}
 \Et_\infty\sum_{|x|=n}\theta_T(x)\ind_{\{Y_x\in A\}}
       \le\kappa^{-1}\Pt_\infty(Y\in A).
 \label{eq:harmonic-domination}
\end{equation}
\end{lemma}

\begin{proof}
Suppose $x$ exists and remove its strict descendants, obtaining a
network $S$. If $S$ is transient, write
$A_x=g_S(o,x)$ and $R_x=g_S(x,x)\ge1$. The latter bound follows
because $x$ is a leaf and its parent edge has resistance one. The
voltage equation at $x$ gives
\begin{equation}
 \theta_T(x)=\frac{A_xC_x^+}{1+R_xC_x^+}
       =U_x\frac{R_xC_x^+}{1+R_xC_x^+},\qquad U_x=\frac{A_x}{R_x}.
 \label{eq:terminal-current}
\end{equation}
Here $U_x$ is the probability of hitting $x$ from $o$ in $S$, so
$U_x\le1$. Since $R_x\ge1$, the last fraction lies between
$C_x^+/(1+C_x^+)=W_x$ and one.

If $S$ is not transient, it consists only of the finite path $[o,x]$:
any side branch would contain an almost surely transient descendant
tree. In this case all the unit current goes through $x$, so set
$U_x=1$. If $x$ does not exist, set $U_x=0$. This proves
\eqref{eq:conditional-flow-comparison} in every case.

Independence of the descendant tree and $\mathcal F_x$ now gives
\begin{align}
 \Et_\infty[\theta_T(x)\mid\mathcal F_x]&\ge\kappa U_x,\nonumber\\
 \Et_\infty[\theta_T(x)\ind_{\{Y_x\in A\}}\mid\mathcal F_x]
 &\le U_x\Pt_\infty(Y\in A)
 \le\kappa^{-1}\Pt_\infty(Y\in A)
                 \Et_\infty[\theta_T(x)\mid\mathcal F_x].
 \label{eq:conditional-mark}
\end{align}
Sum over the countable collection of labels at depth $n$ and use
Tonelli's theorem and conservation of flow.
\end{proof}

\begin{lemma}
\label{lem:harmonic-decay}
There is a deterministic $c>0$, depending only on the law of $Y$,
such that almost surely
\begin{equation}
 \max_{|x|=n}\theta_T(x)\le e^{-cn}
       \quad\text{for all sufficiently large }n.
 \label{eq:harmonic-decay}
\end{equation}
\end{lemma}

\begin{proof}
For an ordinary descendant tree with $Y$ children, let $W_i$ be
the conductance through its $i$th child, including the initial unit
edge, and put
\begin{equation}
 B=\sum_{i=1}^Y\left(\frac{W_i}{C}\right)^2,
 \qquad C=\sum_{i=1}^YW_i,\qquad
 \delta=\Et_\infty[W^2(1-B)].
 \label{eq:split-loss}
\end{equation}
Every $W_i$ is positive. Hence $B\le1$, with strict inequality when
$Y\ge2$, and $0<\delta<1$. At a nonroot vertex $x$, let $B_x$
be the corresponding sum of squared current fractions to its children.
By \eqref{eq:conditional-flow-comparison} and independence from $\mathcal F_x$,
\begin{align}
 \Et_\infty[\theta_T(x)^2(1-B_x)\mid\mathcal F_x]
 &\ge U_x^2\Et_\infty[W^2(1-B)]\nonumber\\
 &=\delta U_x^2
 \ge\delta\Et_\infty[\theta_T(x)^2\mid\mathcal F_x].
 \label{eq:conditional-loss}
\end{align}
Let $H_n=\sum_{|x|=n}\theta_T(x)^2$. The identity
$H_{n+1}=\sum_{|x|=n}\theta_T(x)^2B_x$ and
\eqref{eq:conditional-loss} imply
\begin{equation}
 \Et_\infty H_{n+1}\le(1-\delta)\Et_\infty H_n\quad(n\ge1),\qquad
 \Et_\infty H_n\le(1-\delta)^{n-1},
 \label{eq:flow-second-moment}
\end{equation}
since $H_1\le1$. No moment of the root law is used. Choose $c>0$
with $e^{2c}(1-\delta)<1$. Markov's inequality gives
\begin{equation}
 \Pt_\infty\{\max_{|x|=n}\theta_T(x)>e^{-cn}\}
 \le e^{2cn}(1-\delta)^{n-1}.
 \label{eq:flow-tail}
\end{equation}
The right-hand side is summable, so Borel--Cantelli proves the claim.
\end{proof}

\begin{proposition}
\label{prop:finite-log}
If $\Et_\infty\log(Y+1)<\infty$, then $D_G(T)=1$ for
$\Pt_\infty$-almost every tree.
\end{proposition}

\begin{proof}
Sample $T$ with law $\Pt_\infty$ and then a ray $\xi$ with law
$\theta_T$. Write
$\mathbf Q(dT,d\xi)=\Pt_\infty(dT)\theta_T(d\xi)$.
Equation~\eqref{eq:harmonic-domination} gives, for every
$\varepsilon>0$,
\begin{equation}
 \mathbf Q\{\log\de(\xi_n)>\varepsilon n\}
 \le\kappa^{-1}\Pt_\infty\{\log(Y+1)>\varepsilon n\}.
 \label{eq:degree-tail-on-ray}
\end{equation}
The right-hand side is summable by the logarithmic moment
assumption. Borel--Cantelli for positive rational $\varepsilon$
and Fubini's theorem give $\log\de(\xi_n)=o(n)$ on almost every
harmonic ray of almost every tree. By \eqref{eq:harmonic-decay},
$\log(1/\theta_T(\xi_n))\ge cn$ eventually, so
\eqref{eq:ray-assumption} holds. Lemma~\ref{lem:GW-capacity}
verifies the capacity assumption. The degree criterion
\eqref{eq:ray-assumption}, obtained from
Theorem~\ref{thm:half-analytic}, gives the conclusion.
\end{proof}

\subsubsection{Infinite logarithmic moment}
\label{sec:infinite-log}

We now verify the same cutset condition using flows on subtrees
whose vertices have bounded offspring.
For $M\ge1$, set
\begin{equation}
 p_M=\Pt_\infty(Y\le M),\qquad m_M=\Et_\infty[Y\ind_{\{Y\le M\}}].
 \label{eq:truncation}
\end{equation}
Given a nonroot vertex $v$ with $Y_v\le M$, retain a descendant $x$
if $Y_z\le M$ at every $z\in[v,x]$. The retained tree has offspring
law $\operatorname{Binomial}(\widehat Y_M,p_M)$, where
$\widehat Y_M$ has the law of $Y$ conditioned on $Y\le M$.
Indeed, a retained parent has offspring law $\widehat Y_M$;
each of its children is retained independently with probability
$p_M$, and conditional on retention its offspring has that same
law. Iteration proves the branching property, with mean
$p_M\Et_\infty[Y\mid Y\le M]=m_M$.
Figure~\ref{fig:retained-subtree} illustrates the construction for $M=3$.

\begin{figure}[htbp]
 \centering
 \input{figures/cable_retained_subtree.tex}
 \caption{The retained subtree for $M=3$ is black within the grey
 ambient tree. The vertex $z$ has four children, so it and its whole
 descendant subtree are excluded; $x$ is excluded even though $Y_x=2$.
 The root $v$ has three ambient children and two retained children.}
 \label{fig:retained-subtree}
\end{figure}
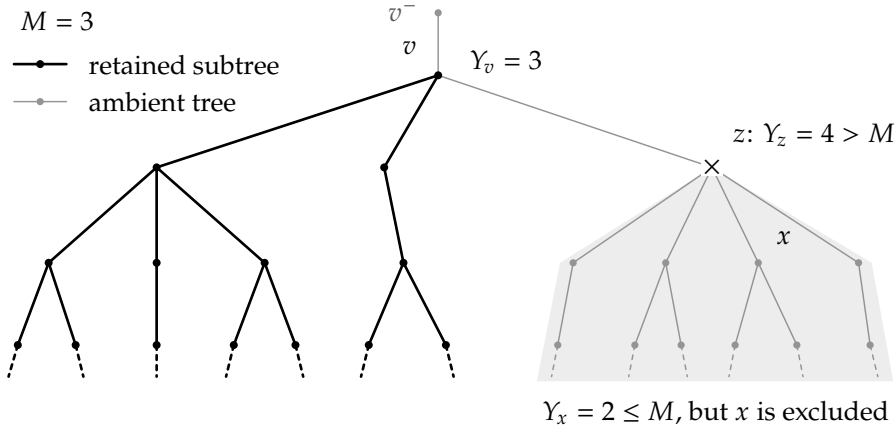

\begin{proposition}
\label{prop:truncated-bound}
For every $M$ with $m_M>1$, almost surely
\begin{equation}
 D_G(T)\ge\frac{\log m_M}{\log(M+1)}.
 \label{eq:truncated-bound}
\end{equation}
\end{proposition}

\begin{proof}
Such a retained tree survives with positive probability, and its
branching number on survival is $m_M$ by~\cite[Corollary~5.10]{LP}.
For $0<s<\log m_M/\log(M+1)$, \eqref{eq:branching-number}
and max-flow/min-cut give a unit flow $\mu$ on a surviving retained
tree rooted at $v$ such that
\begin{equation}
 \mu(x)\le A(M+1)^{-s|x|_v}
 \label{eq:retained-flow}
\end{equation}
for some finite $A$, where $|x|_v$ is the depth from $v$.

Every retained vertex has degree at most $M+1$ in the ambient
pruned tree $T$.
Writing $H_T$ for hitting probabilities in that tree, for adjacent
retained vertices $x,y$,
\begin{equation}
 \rho_T(x,y)=\sqrt{H_T(x,y)H_T(y,x)}
 \ge\frac1{\sqrt{\deg_T(x)\deg_T(y)}}\ge\frac1{M+1}.
 \label{eq:retained-rho}
\end{equation}
Here the first inequality uses the probability of an immediate step.
Multiplicativity of normalized Green functions along a tree path gives
$\rho_T(v,x)\ge(M+1)^{-|x|_v}$. Consequently
$\mu(x)\le A\rho_T(v,x)^s$. Extend the flow along $[o,v]$ with mass
one and set it to zero on all unused branches. Increasing its constant
over this finite initial path gives $\mu(x)\le A'\rho_o(x)^s$
throughout $T$. Every cutset then has weight at least $1/A'$.

It remains to find a surviving retained tree. Let $a_M>0$ be the
probability of the required event at an ordinary vertex, including
$Y_v\le M$. Under \eqref{eq:leafless-law}, the generation size $L_n$
tends to infinity almost surely. Indeed it is nondecreasing, and a
finite eventual value would require a fixed positive number of
individuals to reproduce with exactly one child forever, an event
of probability zero since $\Pt_\infty(Y=1)<1$. Before revealing offspring
at generation $n$, the descendant trees there are independent, so
\begin{equation}
 \Pt_\infty(\text{no suitable vertex anywhere})
       \le\Et_\infty(1-a_M)^{L_n}\longrightarrow0.
 \label{eq:find-subtree}
\end{equation}
The limit follows by bounded convergence. Apply the flow bound and
intersect the resulting events over the countably many eligible $M$.
\end{proof}

If $\Et_\infty\log(Y+1)=\infty$, then
\begin{equation}
 \limsup_{M\to\infty}\frac{\log m_M}{\log(M+1)}=1.
 \label{eq:truncation-exponent}
\end{equation}
For completeness, $m_M\le M$ gives the upper bound. A smaller limsup
would imply $m_M\le A M^a$ for some $a<1$ and all large $M$.
Then
\begin{equation}
 \Pt_\infty(2^k<Y\le2^{k+1})\le2^{-k}m_{2^{k+1}}
       \le A'2^{-(1-a)k},
 \label{eq:dyadic-tail}
\end{equation}
which, after multiplication by $k+1$ and summation, gives a finite
logarithmic moment, a contradiction. Infinite logarithmic moment also
implies $m_M\to\infty$, so the relevant cutoffs have $m_M>1$.
\begin{proof}[Proof of Theorem~\ref{thm:GW}]
Conditional on $\mathsf S$, the pruned tree has law $\Pt_\infty$
as described in \eqref{eq:surviving-law}. If $Y$ has finite
logarithmic moment, Proposition~\ref{prop:finite-log} gives $D_G=1$.
Otherwise, \eqref{eq:truncated-bound} and
\eqref{eq:truncation-exponent} give $D_G\ge1$, while
Lemma~\ref{lem:GW-capacity} and Theorem~\ref{thm:half-analytic}
give the reverse inequality. Lemma~\ref{lem:pruning} transfers both
$D_G=1$ and \eqref{eq:capacity-condition} to the original tree.
Theorem~\ref{thm:half-analytic} now gives the critical intensity and
non-percolation at one-half. Equations~\eqref{eq:frostman},
\eqref{eq:survival} and \eqref{eq:full-connection} give positive
root percolation probability at every $\alpha>1/2$ on the same
full-probability set of trees.
\end{proof}

For an explicit law with infinite logarithmic moment, take
\begin{equation}
 \nu(2^k)=\frac1{k(k+1)}\quad(k\ge1),\qquad
 \nu(j)=0\quad\text{otherwise}.
 \label{eq:heavy-tail-example}
\end{equation}
There is no extinction for this law, so $\bar\nu=\nu$.
Its logarithmic moment and every positive power moment are
infinite. Moreover,
\begin{equation}
 m_{2^k}\ge\frac{2^k}{k(k+1)},
 \label{eq:heavy-tail-truncated-mean}
\end{equation}
which makes \eqref{eq:truncation-exponent} immediate.

\printbibliography
\end{document}

%% file: figures/cable_green_boundary.tex
\begin{tikzpicture}[
 x=1.12cm,y=0.96cm,
 font=\small,
 line cap=round,line join=round,
 edge/.style={draw=black!60,line width=0.5pt},
 ray/.style={draw=black,line width=0.9pt},
 shared/.style={draw=black,line width=1.8pt},
 continuation/.style={dash pattern=on 1.8pt off 2pt},
 vertex/.style={circle,fill=black,inner sep=0pt,minimum size=2.7pt},
 cut vertex/.style={circle,draw=black,fill=white,line width=0.6pt,
                    inner sep=0pt,minimum size=7.5pt}
]
 \coordinate (o) at (0,4.15);
 \coordinate (m) at (0.72,3.4);
 \coordinate (a) at (1.5,2.6);
 \coordinate (x) at (-2.8,2.6);
 \coordinate (b) at (0.05,1.5);
 \coordinate (c) at (2.9,1.5);
 \coordinate (u) at (2.15,0.55);
 \coordinate (v) at (3.7,0.55);
 \coordinate (xl) at (-4.0,1.35);
 \coordinate (xm) at (-2.8,1.12);
 \coordinate (xr) at (-1.6,1.35);
 \coordinate (xll) at (-4.2,0.55);
 \coordinate (xlr) at (-3.55,0.55);
 \coordinate (xmm) at (-2.8,0.5);
 \coordinate (xrl) at (-1.85,0.5);
 \coordinate (xrr) at (-1.3,0.5);
 \coordinate (bl) at (-0.58,0.6);
 \coordinate (br) at (0.72,0.6);

 \fill[black!7] (x) -- (-4.16,1.35) -- (-4.5,-0.05)
   -- (-1.12,-0.05) -- (-1.44,1.35) -- cycle;

 \draw[edge] (o) -- (x);
 \draw[edge] (x) -- (xl) (x) -- (xm) (x) -- (xr);
 \draw[edge] (xl) -- (xll) (xl) -- (xlr) (xm) -- (xmm);
 \draw[edge] (xr) -- (xrl) (xr) -- (xrr);
 \draw[edge,continuation]
   (xll) -- (-4.32,0) (xlr) -- (-3.4,0) (xmm) -- (-2.8,0)
   (xrl) -- (-1.95,0) (xrr) -- (-1.18,0);

 \draw[edge] (b) -- (br) (c) -- (v);
 \draw[edge,continuation] (br) -- (1.05,0);
 \draw[edge] (v) -- (3.38,0.22) (v) -- (4.02,0.22);
 \draw[edge,continuation]
   (3.38,0.22) -- (3.16,-0.04) (4.02,0.22) -- (4.25,-0.04);

 \draw[shared] (o) -- (m) -- (a);
 \draw[ray] (a) -- (b) -- (bl) (a) -- (c) -- (u);
 \draw[ray,continuation] (bl) -- (-0.98,0) (u) -- (1.94,0);

 \foreach \name in {x,b,u,v}
   \node[cut vertex] at (\name) {};
 \foreach \name in {o,m,a,x,b,c,u,v,xl,xm,xr,xll,xlr,xmm,xrl,xrr,bl,br}
   \node[vertex] at (\name) {};

 \node[above left=2pt] at (o) {$o$};
 \node[left=5pt] at (x) {$x$};
 \node[above right=3pt] at (a) {$a=\xi\wedge\eta$};
 \node[below=4pt] at (-0.98,0) {$\xi$};
 \node[below=4pt] at (1.94,0) {$\eta$};
 \node[below=5pt] at (-2.8,-0.05) {$\partial T_x$};
 \node[anchor=west] at (2.0,3.92)
   {$\delta_G(\xi,\eta)=\rho_o(a)$};

 \node[cut vertex] at (4.62,1.5) {};
 \node[vertex] at (4.62,1.5) {};
 \node[anchor=west] at (4.82,1.5) {cutset $\Pi$};
\end{tikzpicture}

%% file: figures/cable_ray_capacity.tex
\begin{tikzpicture}[
 x=0.97cm,y=0.88cm,
 font=\small,
 line cap=round,line join=round,
 edge/.style={draw=black!60,line width=0.5pt},
 ray/.style={draw=black,line width=0.9pt},
 segment/.style={draw=black,line width=1.8pt},
 wire/.style={draw=black,line width=0.9pt},
 continuation/.style={dash pattern=on 1.8pt off 2pt},
 vertex/.style={circle,fill=black,inner sep=0pt,minimum size=2.7pt},
 branch brace/.style={draw=black!65,line width=0.4pt,
   decorate,decoration={brace,amplitude=3pt}}
]
 \node at (1.9,5.18) {(a) Ray segment};
 \node at (9.9,5.18) {(b) After wiring};

 \coordinate (o) at (1.25,4.3);
 \coordinate (xi1) at (1.25,3.1);
 \coordinate (xinminus) at (1.25,1.4);
 \coordinate (xin) at (1.25,0.1);
 \draw[segment] (o) -- (xi1) -- (1.25,2.7);
 \node at (1.25,2.26) {$\vdots$};
 \draw[segment] (1.25,1.84) -- (xinminus) -- (xin);

 \foreach \level/\index in {4.3/0,3.1/1,1.4/{n-1}} {
   \draw[edge] (1.25,\level) -- (2.0,{\level+0.25})
               (1.25,\level) -- (2.0,{\level-0.25});
   \draw[edge] (2.0,{\level+0.25}) -- (2.42,{\level+0.45})
               (2.0,{\level+0.25}) -- (2.42,{\level+0.18})
               (2.0,{\level-0.25}) -- (2.42,{\level-0.18})
               (2.0,{\level-0.25}) -- (2.42,{\level-0.45});
   \draw[edge,continuation]
     (2.42,{\level+0.45}) -- (2.72,{\level+0.55})
     (2.42,{\level+0.18}) -- (2.72,{\level+0.12})
     (2.42,{\level-0.18}) -- (2.72,{\level-0.12})
     (2.42,{\level-0.45}) -- (2.72,{\level-0.55});
   \node[vertex] at (2.0,{\level+0.25}) {};
   \node[vertex] at (2.0,{\level-0.25}) {};
   \draw[branch brace] (2.94,{\level+0.56}) -- (2.94,{\level-0.56});
   \node[anchor=west] at (3.08,\level) {$s_{\index}(\xi)$};
 }

 \fill[black!7] (xin) -- (0.4,-0.35) -- (-0.05,-1.25)
   -- (2.65,-1.25) -- (2.1,-0.35) -- cycle;
 \draw[edge] (xin) -- (0.58,-0.4) (xin) -- (1.92,-0.4);
 \draw[edge] (0.58,-0.4) -- (0.2,-0.88)
             (0.58,-0.4) -- (0.86,-0.88)
             (1.92,-0.4) -- (1.64,-0.88)
             (1.92,-0.4) -- (2.3,-0.88);
 \draw[edge,continuation]
   (0.2,-0.88) -- (0.0,-1.2) (0.86,-0.88) -- (0.98,-1.2)
   (1.64,-0.88) -- (1.52,-1.2) (2.3,-0.88) -- (2.52,-1.2);
 \draw[ray] (xin) -- (1.25,-0.6);
 \draw[ray,continuation] (1.25,-0.6) -- (1.25,-1.22);
 \node[vertex] at (0.58,-0.4) {};
 \node[vertex] at (1.92,-0.4) {};
 \node[vertex] at (1.25,-0.6) {};
 \draw[branch brace] (2.94,0.03) -- (2.94,-1.25);
 \node[anchor=west] at (3.08,-0.61) {$C_{\xi_n}^{+}$};
 \node[below=3pt] at (1.25,-1.22) {$\xi$};

 \foreach \name in {o,xi1,xinminus,xin}
   \node[vertex] at (\name) {};
 \node[left=5pt] at (o) {$o=\xi_0$};
 \node[left=5pt] at (xi1) {$\xi_1$};
 \node[left=5pt] at (xinminus) {$\xi_{n-1}$};
 \node[left=5pt] at (xin) {$\xi_n$};

 \draw[-{Stealth[length=5pt,width=4pt]},line width=0.6pt]
   (4.5,2.25) -- (6.3,2.25)
   node[midway,above=4pt] {wire $[o,\xi_n]$};

 \draw[wire] (9.9,4.2) -- (9.9,3.5);
 \node[vertex,minimum size=4.5pt] at (9.9,4.2) {};
 \node[above=4pt] at (9.9,4.2) {$[o,\xi_n]$};
 \draw[wire] (7.15,3.5) -- (12.65,3.5);
 \draw[wire] (7.15,0.15) -- (12.65,0.15);
 \foreach \position/\conductance in
   {7.15/{s_0(\xi)},8.75/{s_1(\xi)},11.05/{s_{n-1}(\xi)},12.65/{C_{\xi_n}^{+}}} {
   \draw[wire] (\position,3.5) -- (\position,0.15);
   \node[draw=black,fill=white,line width=0.6pt,
         minimum width=1.35cm,minimum height=0.72cm,inner sep=3pt]
     at (\position,1.9) {$\conductance$};
   \node[vertex] at (\position,3.5) {};
   \node[vertex] at (\position,0.15) {};
 }
 \node at (9.96,1.9) {$\cdots$};
 \draw[wire] (9.9,0.15) -- (9.9,-0.38);
 \node[vertex] at (9.9,-0.38) {};
 \node[below=3pt] at (9.9,-0.38) {$\infty$};
 \node at (9.9,-1.48)
   {$\operatorname{cap}([o,\xi_n])=\sum_{j=0}^{n-1}s_j(\xi)+C_{\xi_n}^{+}$};
\end{tikzpicture}

%% file: figures/cable_retained_subtree.tex
\begin{tikzpicture}[
 x=1.02cm,y=0.87cm,
 font=\small,
 line cap=round,line join=round,
 ambient/.style={draw=black!42,line width=0.55pt},
 retained/.style={draw=black,line width=1pt},
 continuation/.style={dash pattern=on 1.8pt off 2pt},
 ambient vertex/.style={circle,fill=black!42,inner sep=0pt,minimum size=2.7pt},
 retained vertex/.style={circle,fill=black,inner sep=0pt,minimum size=3pt}
]
 \coordinate (parent) at (0,5.15);
 \coordinate (v) at (0,4.2);
 \coordinate (a) at (-3.65,2.8);
 \coordinate (b) at (-0.7,2.8);
 \coordinate (z) at (3.55,2.8);
 \coordinate (aa) at (-5.05,1.35);
 \coordinate (ab) at (-3.65,1.35);
 \coordinate (ac) at (-2.25,1.35);
 \coordinate (bb) at (-0.45,1.35);
 \coordinate (zl) at (1.75,1.35);
 \coordinate (zm) at (2.95,1.35);
 \coordinate (x) at (4.15,1.35);
 \coordinate (zr) at (5.45,1.35);
 \coordinate (aaone) at (-5.45,0.1);
 \coordinate (aatwo) at (-4.7,0.1);
 \coordinate (abone) at (-3.65,0.1);
 \coordinate (acone) at (-2.65,0.1);
 \coordinate (actwo) at (-1.85,0.1);
 \coordinate (bbone) at (-0.9,0.1);
 \coordinate (bbtwo) at (0.1,0.1);
 \coordinate (zlone) at (1.55,0.1);
 \coordinate (zmone) at (2.55,0.1);
 \coordinate (zmtwo) at (3.15,0.1);
 \coordinate (xone) at (3.85,0.1);
 \coordinate (xtwo) at (4.65,0.1);
 \coordinate (zrone) at (5.6,0.1);

 \fill[black!7] (z) -- (1.58,1.35) -- (1.28,-0.45)
   -- (5.9,-0.45) -- (5.62,1.35) -- cycle;

 \draw[ambient] (parent) -- (v);
 \foreach \treeparent/\treechild in {
   v/a,v/b,v/z,a/aa,a/ab,a/ac,b/bb,z/zl,z/zm,z/x,z/zr,
   aa/aaone,aa/aatwo,ab/abone,ac/acone,ac/actwo,bb/bbone,bb/bbtwo,
   zl/zlone,zm/zmone,zm/zmtwo,x/xone,x/xtwo,zr/zrone}
   \draw[ambient] (\treeparent) -- (\treechild);
 \foreach \treeparent/\treechild in {
   v/a,v/b,a/aa,a/ab,a/ac,b/bb,
   aa/aaone,aa/aatwo,ab/abone,ac/acone,ac/actwo,bb/bbone,bb/bbtwo}
   \draw[retained] (\treeparent) -- (\treechild);

 \foreach \branchposition/\branchshift in {
   -5.45/-0.12,-4.7/0.1,-3.65/0,-2.65/-0.1,-1.85/0.1,-0.9/-0.1,0.1/0.1}
   \draw[retained,continuation] (\branchposition,0.1)
     -- ({\branchposition+\branchshift},-0.38);
 \foreach \branchposition/\branchshift in {
   1.55/-0.08,2.55/-0.08,3.15/0.08,3.85/-0.08,4.65/0.08,5.6/0.1}
   \draw[ambient,continuation] (\branchposition,0.1)
     -- ({\branchposition+\branchshift},-0.38);

 \foreach \name in {parent,zl,zm,x,zr,zlone,zmone,zmtwo,xone,xtwo,zrone}
   \node[ambient vertex] at (\name) {};
 \foreach \name in {v,a,b,aa,ab,ac,bb,aaone,aatwo,abone,acone,actwo,bbone,bbtwo}
   \node[retained vertex] at (\name) {};
 \node[fill=white,inner sep=0.7pt] at (z) {$\times$};

 \node[left=4pt,text=black!65] at (parent) {$v^-$};
 \node[above left=4pt] at (v) {$v$};
 \node[anchor=west] at (0.28,4.4) {$Y_v=3$};
 \node[above right=4pt] at (z) {$z$: $Y_z=4>M$};
 \node[above right=3pt] at (x) {$x$};
 \node at (3.6,-0.93) {$Y_x=2\le M$, but $x$ is excluded};

 \node[anchor=west] at (-5.55,5.02) {$M=3$};
 \draw[retained] (-5.5,4.37) -- (-4.82,4.37);
 \node[retained vertex] at (-5.16,4.37) {};
 \node[anchor=west] at (-4.66,4.37) {retained subtree};
 \draw[ambient] (-5.5,3.8) -- (-4.82,3.8);
 \node[ambient vertex] at (-5.16,3.8) {};
 \node[anchor=west] at (-4.66,3.8) {ambient tree};
\end{tikzpicture}